\documentclass[10pt]{article}
\usepackage[top=4.2cm, bottom=4.2cm, left=4.0cm, right=4.0cm]{geometry}
\usepackage[colorlinks = true,
            linkcolor = blue,
            urlcolor  = blue,
            citecolor = blue,
            anchorcolor = blue]{hyperref}
            
\usepackage[english]{babel}
\usepackage{enumerate,hyperref}
\usepackage{enumitem}
\usepackage{amssymb}
\usepackage{amsfonts}
\usepackage{graphicx}
\usepackage{amsmath}
\usepackage{amssymb}
\usepackage{latexsym}
\usepackage{amsthm}
\usepackage{appendix}
\usepackage{eufrak}
\usepackage[all]{xy}

\usepackage{todonotes} %to help annotate the working paper
\usepackage{fancyhdr}
\usepackage{comment}

\usepackage[linesnumbered,vlined,boxed,commentsnumbered]{algorithm2e}
\usepackage{caption}

\usepackage{mathtools}  % need for `show only references'
\usepackage{dsfont}
\usepackage{xspace}

\makeatletter
\def\timenow{\@tempcnta\time
\@tempcntb\@tempcnta
\divide\@tempcntb60
\ifnum10>\@tempcntb0\fi\number\@tempcntb
:\multiply\@tempcntb60
\advance\@tempcnta-\@tempcntb
\ifnum10>\@tempcnta0\fi\number\@tempcnta}
\makeatother

\numberwithin{equation}{section}
\newcommand{\beqa}{\begin{eqnarray}}
\newcommand{\eeqa}{\end{eqnarray}}
\newcommand\bal{\begin{aligned}}
\newcommand\eal{\end{aligned}}
\newcommand{\beql}[1]{\begin{equation}\label{#1}\bal}
\newcommand{\eeql}{\eal\end{equation}}

\newcommand{\cA}{\mathcal{A}}
\newcommand \esssup{\mbox{\upshape ess}\sup}

\newcommand{\cN}{\mathcal{N}}
\newcommand \cP{{\cal P}}

\newcommand{\calL}{\mathcal{L}}

\newcommand{\N}{\mathbb{N}}
\newcommand{\R}{\mathbb{R}} 
\newcommand{\un}{\ind}
\newcommand\ind{\mathds{1}}

\newcommand \PE{\mathbb{E}}

\newcommand \Esp[1]{\mathbb{E}\left[#1\right]}
\newcommand \Espcon[2]{\mathbb{E}^{#1}\left[#2\right]}
\newcommand \PP{\mathbb{P}}
\newcommand \PPro[1]{\PP\left(#1\right)}
\newcommand \PProcon[2]{\PP^{\,#1}\left[#2\right]}

\newcommand \es{\mathrm{es}}\renewcommand \es{s}

\newcommand \Radcemp{\mathcal{R}_{emp}}
\newcommand \Radcave{\mathcal{R}_{ave}}

\newcommand\alo[1]{{\mathbf \xi(#1)}}

\newcommand\stax{S}
\def\proof{\noindent {\it Proof. $\, $}}\def\proof{\noindent{\it\textbf{Proof}. $\, $}}
\def\finproof{\hfill\rule{4pt}{6pt}}

\newtheorem{thm}{Theorem}[section]
\newtheorem{lem}{Lemma}[section]
\newtheorem{pro}{Proposition}[section]
\newtheorem{cor}{Corollary}[section]
\newtheorem{conj}{Conjecture}[section]

\newtheorem{rem}{Remark}[section]
\newtheorem{com}{Comments}[section]

\newtheorem{nota}{Notation}[section]
\newtheorem{dfn}{Definition}[section]
\newtheorem{hyp}{Assumption}[section]
\theoremstyle{definition}
\newtheorem{exa}{Example}[section]
\newcommand{\bt}{\begin{theorem}}\newcommand{\et}{\end{theorem}}
\newcommand{\bl}{\begin{lem}}\newcommand{\el}{\end{lem}}
\newcommand{\bp}{\begin{pro}}\newcommand{\ep}{\end{pro}}
\newcommand{\bcor}{\begin{cor}}\newcommand{\ecor}{\end{cor}}
\newcommand{\bconj}{\begin{conj}}\newcommand{\econj}{\end{conj}}

\newcommand{\bd}{\begin{defi} \rm 
}
\newcommand{\ed}{\finproof\end{defi} }
\newcommand{\eds}{\end{defi} }
\newcommand{\brem }{\begin{rem} \rm }\newcommand{\erem }{\finproof\end{rem}}
\newcommand{\bcom}{\begin{com} \rm }\newcommand{\ecom }{\end{com}}
\newcommand{\brems }{\begin{rem} \rm }\newcommand{\erems }{\end{rem}}
\newcommand{\bex}{\begin{ex} \rm }\newcommand{\eex}{\finproof\end{ex}}
\newcommand{\bno}{\begin{nota} \rm }\newcommand{\eno}{\finproof\end{nota}}\newcommand{\enos}{\end{nota}}
\newcommand{\eexs}{\end{ex} }
\newcommand{\bhyp}{\begin{hyp} %\rm 
}\newcommand{\ehyp}{\end{hyp}}

\def\eee{\end{document}}

\renewcommand \PE{{\rm E}}

\def\qqq{\quad\quad\quad}

\def\bb{\begin{block}}
\def\eb{\end{block}} 
\renewcommand \PE{{\rm E}}
\renewcommand \Esp[1]{\PE\left[#1\right]}
\renewcommand \Espcon[2]{\PE^{#1}\left[#2\right]}
\renewcommand \PP{{\rm P}}
\renewcommand \PProcon[2]{\PP^{\,#1}\left[#2\right]}
\newcommand{\cF}{\mathcal{F}}
\newcommand{\cG}{\mathcal{G}}
\newcommand{\cH}{\mathcal{H}}

\newcommand{\eqdef}{\mathrel{\mathop:}=}\def\eqd{\eqdef}
\newcommand{\defeq}{\mathrel{{=}{\mathop:}}}
\def\finproof{\rule{4pt}{6pt}}\def\fe{~\finproof}
\usepackage{xr}%for cross references
\title{{\bf Confidence intervals for nonparametric regression}\\} \author{\large David Barrera\thanks{Email: {\tt j.barrerac@uniandes.edu.co}. Departamento de Matem\'{a}ticas, Universidad de los Andes, Cra 1 \# 18a-12, Edificio H. Bogot\'{a}, Colombia. Postal code: 111711.} \thanks{The author is thankful for the comments and suggestions following from the reading by the co-authors in \cite{barcregobngusaa}, which helped to improve the presentation and to clean up several errors.
}  
} \date{ }

\begin{document} \maketitle
\begin{abstract}
We demonstrate and discuss nonasymptotic bounds in probability for the cost of  a regression scheme with a general loss function from the perspective of the Rademacher theory,  and for the optimality with respect to the average $L^{2}$-distance to the underlying conditional expectations of  least squares regression outcomes from the perspective of the Vapnik-Chervonenkis theory. 

The results follow from an analysis involving  independent but possibly nonstationary training samples and can be extended, in a manner that we explain and illustrate, to relevant cases in which the training sample exhibits dependence.

\vspace{0.2cm}
\end{abstract}
{
  \hypersetup{linkcolor=black}
  \tableofcontents
}
\section{Introduction}

This paper is a companion to \cite{barcregobngusaa}, which proposes new methods and error bounds in the nonparametric, distribution-free setting for the approximation of conditional quantiles and expected shortfalls. By the latter, we mean functions 
\begin{align}
\label{equdefconquaconexpsho}
x\mapsto q_{\alpha}(x),&\qquad x\mapsto \es_{\alpha}(x) 
\end{align}
defined on a Polish space $S$ where, for a pair of random variables $(X,Y)\in S\times \R$, and under appropriate regularity hypotheses on the conditional distribution of $Y$ given $X=x$\footnote{For instance its absolute continuity and strictly positive density for $\PP_{X}-$a.e. $x\in S$.}, $q_{\alpha}(\cdot)$ and $\es_{\alpha}(\cdot)$ are characterized by the properties
\begin{align}
\label{equchaconquaconexpsho}
\PProcon{x}{Y\leq q_{\alpha}(x)}=\alpha,&\qquad (1-\alpha)\es_{\alpha}(x)=\Espcon{x}{Y\un_{\{Y\geq q_{\alpha}(x)\}}}&\qquad\mbox{ for $\PP_{X}-$a.e. $x$},
\end{align}
where $\PProcon{x}{\,\cdot\,}$ and $\Espcon{x}{\,\cdot\,}$ denote  the conditional probability of $Y$ given $X=x$ and the expectation with respect to $\PProcon{x}{\,\cdot\,}$.

The approximation of the functions \eqref{equdefconquaconexpsho} is an important problem in statistical inference in general \cite{koenker2017quantile} and in particular nowadays within the context of machine learning methods \cite{barcregobngusaa}.

The present paper arose from the necessity of producing bounds in probability applicable to the convergence analysis of the methods in \cite{barcregobngusaa}, concretely to the approximation $\hat{q}_{\alpha}(\cdot)$ of $q_{\alpha}(\cdot)$ through a scheme based on the ``pinball'' or ``tilted'' loss\footnote{See \cite[eqn. (\ref*{vares-equappvar})]{barcregobngusaa}.} and to the least squares scheme proposed to produce and approximation $\hat{r}_{\hat{q}_{\alpha}}(\cdot)$ of the function $r_{\alpha}(\cdot):=\es_{\alpha}(\cdot)-q_{\alpha}(\cdot)$ given any approximation $\hat{q}_{\alpha}(\cdot)$ of $q_{\alpha}(\cdot)$\footnote{See \cite[eqn. (\ref*{vares-equappexpsho})]{barcregobngusaa}.}.

More precisely, it turns out that the analysis (and the construction) of the regression scheme for the approximation of 
$\es_{\alpha}(\cdot)$ proposed in \cite{barcregobngusaa} is possible via a rather natural continuation of the arguments presented in the papers \cite{bargob19} and \cite{bargob19seq} in which nonparametric, distribution-free error bounds associated to learning  schemes are discussed in nonstationary settings mostly in the context of the ``Vapnik-Chervonenkis'', or ``VC'' theory (\cite{gyor:kohl:krzy:walk:02, vap}), complemented by a  discussion on  learning bounds via the Rademacher theory (see \cite{wol,mohri2018foundations}), applicable to the pinball loss associated to the scheme in \cite{barcregobngusaa} used to approximate $q_{\alpha}(\cdot)$. 
The results following from this task, which go beyond the application in \cite{barcregobngusaa}, and are therefore of general interest, constitute  the subject of this paper.

From this general perspective, an alternative motivation for the present paper arises from the increasing demand for rigorous expositions on  nonasymptotic bounds associated to regression schemes in which the learning sample may be nonstationary.\footnote{The webpage \cite{sha} contains a list of sources dedicated to several instances of this scenario.} The continuity with \cite{bargob19} and \cite{bargob19seq} is rather evident in this context: \cite{bargob19} is basically an extension of the analysis in \cite[Chapter 11, 12]{gyor:kohl:krzy:walk:02} aimed to provide concentration inequalities
associated to least squares which have ``the right rate'' (and constants tighter than those in \cite{gyor:kohl:krzy:walk:02}) that are
valid also in the independent, nonstationary case; \cite{bargob19seq} continues these developments by applying coupling ideas that permit, in particular, a set of weak error bounds for least squares regression schemes with $\beta$-mixing learning samples\footnote{See e.g. \cite[Theorem 20 and Propositions 28, 29]{bargob19seq}.}. The present paper continues this analysis towards the corresponding bounds in probability, complemented by an application of the coupling idea in  \cite{bargob19seq} to the bounds obtained for general loss functions from the Rademacher theory.

We remark that, while we do not think that this paper contains any essential contribution to the bounds developed via the Rademaccher theory (beyond their extension to dependent cases),  we believe that the inclusion of the corresponding survey-like Section \ref{p:tools} below is justified by at least two reasons: first, the author has not been able to find a presentation of these results which is concise enough to serve as a single source  for the analysis performed  in \cite{barcregobngusaa}: the presentation below provides such a discussion departing from first principles ``modulo folklore''; second, in spite of not being an essentially difficult task, there seems to be a lack of references presenting the Rademacher theory (with independence) in a nonstationary context. We hope therefore that this paper serves to motivate a broadening of perspective in this direction, which is further justified by the emergence in practice of learning scenarios in which the training sample may not be i.i.d.. 

\subsubsection*{Organization of the paper}
The paper is organized as follows: Section \ref{secprecon} presents the notation and conventions to be used, and a general remark on the extension of concentration inequalities from independent to dependent samples. Section \ref{p:tools} is a relatively self-contained presentation of some concentration inequalities via the Rademacher theory and of their application to upper bounds on the probability of large deviations associated to the empirical process, illustrated at the end with examples. Finally, Section \ref{secgenvc} presents large deviation estimates via the VC theory akin to those in Section \ref{p:tools} but specialized to the deviation of the quadratic loss, and interprets them in terms of the optimality of the average $L^{2}-$distance between the function obtained by empirical minimization and the conditional expectations of the responses given the covariates.

\section{Preliminary conventions and observations}
\label{secprecon}

We begin in Section \ref{s:app} by explaining the notation used in the paper: the conventions introduced in this regard are important for a concise presentation of the proofs. Then we present in Section \ref{secremdep} some coupling results allowing to extend our bounds, which are demonstrated under the assumption of independent training samples, to the  dependent case (these ideas are applied later, in sections \ref{p:tools} and \ref{secgenvc}, for remarks \ref{rembetmixrad} and \ref{remdepbouvc}).
 \subsection{Notation and setup\label{s:app}}

The following notation and conventions will be used in what follows:
 
 \medskip

{\itshape\bfseries State spaces, sequential functions.} Given a Polish space $S$, $\calL_{S}$ denotes the set of Borel measurable functions $S\to \R$. If  $(S_{k})_{k=1}^{n}$ is a sequence of Polish spaces, any subset
 \begin{align}
 \cH_{1:n}\subset \prod_{k=1}^{n}\calL_{S_{k}}\defeq \calL_{S_{1:n}}^{\otimes}
 \end{align} 
 will be called a {\it  family of sequential  functions on
 \begin{align}
 \prod_{k=1}^{n}S_{k}\defeq S_{1:n}^{\otimes}.
 \end{align}
 }  
 
 \brem
 \label{remidedia}
 If $S_{1:n}^{\otimes}=S^{n}$ and $\cH\subset \calL_{S}$, we will use often the identification $\cH\equiv diag(\cH)_{1:n}$ where 
\begin{align}
\label{equdefdiafam}
diag(\cH)_{1:n}\eqdef \{(h,\dots,h):h\in \cH\}\subset (\calL_{S})^{n},
\end{align}
in order to ``treat'' $\cH$ as a subset of $(\calL_{S})^{n}$.
\erems

\noindent{\itshape\bfseries  Random elements. }A {random element} is a Borel measurable function $Z:\Omega\to S$ where $(\Omega,\cA,\PP)$ is a probability space and $S$ is some Polish space (which will be clear from the context). We will use the usual notation
\begin{align}
\label{equdefpreima}
Z\in B:=Z^{-1}(B):=\{\omega\in \Omega:Z(\omega)\in B\}\in \cA
\end{align}
where $B$ is a Borel element of $S$. If $S=\R$ we will refer to $Z$ as a {\it random variable}, and we will denote by 
\begin{align}
||Z||_{\PP,\infty}:=\inf\{z\in \mathbb{R}:\PPro{|Z|>z}=0\}
\end{align}
the {\it $L^{\infty}_{\PP}$ norm of $Z$}, with the convention $\inf\emptyset=\infty$, and by
\begin{align}
||Z||_{\PP,p}=\left(\int_{\Omega}|Z(\omega)|^{p}\PP(d\omega)\right)^{1/p}
\end{align}
the {\it  $L^{p}_{\PP}$ norm of $Z$} ($p\in [0,\infty)$), where the integral takes the value $\infty$ if $|Z|^{p}$ is not $\PP-$integrable.

All the random elements below will be assumed to be defined {\it in the same fixed probability space}, whose existence can be verified {\it a posteriori} by standard measure-theoretical methods. This is easy since all of our arguments will involve only finitely many random elements. 

 \brem
Let $\cH_{1:n}\subset \calL_{S_{1:n}}^{\otimes}$. If $H_{1:n}:S_{1:n}^{\otimes}\to [0,\infty]^{n}$ is the sequential function defined by 
\begin{align}
\label{equdefseqfuncaph}
H_{k}(z_{k})\eqd \sup_{h_{1:n}\in \cH_{1:n}}|h_{k}(z_{k})|,&\qquad 1\leq k\leq n, 
\end{align}
 if $(h_{1:n}^{(l)})$ is the family in \eqref{fampoimes}, and if $Z_{1:n}$ is a random element of $S_{1:n}^{\otimes}$, then the inequalities
\begin{align}
\label{equinesupasone}
H_{k}(Z_{k})
=\sup_{l\in \N}|h_{k}^{(l)}(Z_{k})|\leq \sup_{h_{1:n}\in \cH_{1:n}}||h_{k}(Z_{k})||_{\PP,\infty}, &\qquad \PP-a.s.,
\end{align}  
valid for $1\leq k\leq n$,  imply that
\begin{align}
\label{equexcsupnor}
||H_{k}(Z_{k})||_{\PP,\infty}= \sup_{h_{1:n}\in \cH_{1:n}}||h_{k}(Z_{k})||_{\PP,\infty}.
\end{align}
(``$\leq$'' follows from \eqref{equinesupasone}, and ``$\geq$'' is true under no condition). 

If the components of $Z_{1:n}$ are even more  independent, then\footnote{If $X_{1:n}$ is independent then $||\,|X_{1:n}|_{n,p}\,||_{\PP,\infty}=|\,||X_{1:n}||_{\PP,\infty}\,|_{n,p}$: ``$\leq$'' is elementary, and ``$\geq$'' follows from an approximation argument: for every $\epsilon>0$, the event $\{|\,||X_{1:n}||_{\PP,\infty}\,|_{n,p}< |X_{1:n}|_{n,p}+\epsilon \}$ has (by independence) positive probability. }
\begin{align}
\label{equequl2norlinfnorlinfnorl2nor}
||\,|H_{1:n}(Z_{1:n})|_{{n,p}}\, ||_{\PP,\infty}= |\,||H_{1:n}(Z_{1:n})||_{\PP,\infty}\,|_{{n,p}}.
\end{align}
\erems

 \medskip
 
\noindent{\itshape\bfseries  Operations with sequential functions.} If $z_{1:n}$ is an element of $S_{1:n}^{\otimes}$, and if $h_{1:n}$ is a sequential function on $S_{1:n}^{\otimes}$, we will denote by $ h_{1:n}(z_{1:n})$ the vector
 \begin{align}
 \label{equopecomwis}
 h_{1:n}(z_{1:n})\eqd (h_{k}(z_{k}))_{k=1}^{n}\in \R^n,
 \end{align}
 and we will operate with sequential functions in a component-wise manner, with posible multiplication by scalars. Thus if $(a,g_{1:n},h_{1:n})\in \R\times\calL_{S_{1:n}}^{\otimes}\times \calL_{S_{1:n}}^{\otimes}$ 
\begin{align}
(a g_{1:n}+h_{1:n})(z_{1:n})=a g_{1:n}(z_{1:n})+a h_{1:n}(z_{1:n})=(ag_{k}(z_{k})+h_{k}(z_{k}))_{k=1}^{n}\\
( g_{1:n}h_{1:n})(z_{1:n})= g_{1:n}(z_{1:n})h_{1:n}(z_{1:n})=(g_{k}(z_{k})h_{k}(z_{k}))_{k=1}^{n}
\end{align} 
for every $z_{1:n}\in S_{1:n}^{\otimes}$. We will also make use of the scalar product
\begin{align}
a_{1:n}\cdot b_{1:n}=\sum_{k=1}^{n}a_{k}b_{k}
\end{align}
where the state space of $a_{1:n}$ or $b_{1:n}$ will be clear from context. 

For any two $\cG_{1:n}\cup \cH_{1:n}\subset \calL_{S_{1:n}}^{\otimes}$, we will use the notation
 \begin{align}
 \label{equdefdirsum}
 \cG_{1:n}+\cH_{1:n}\eqdef \{g_{1:n}+h_{1:n}:h_{1:n}^{(k)}\in \cH_{1:n}^{(k)}\}
 \end{align}
 for the direct sum of  $\cG_{1:n}$ and $\cH_{1:n}$. A similar interpretation defines $a\cH_{1:n}$ (where $a$ is a scalar), $\cG_{1:n}\cH_{1:n}$, and $\cG_{1:n}\cdot\cH_{1:n}$.

 \medskip
 
\noindent{\itshape\bfseries  Operations via a componentwise defined functional.} We will define
\begin{align}
F(a_{1:n})\eqdef (F(a_{1}),\dots, F(a_{n}))
\end{align}
whenever ``$F$'' is an operator well defined on each $a_{k}$ (this is in harmony with the identification $F=(F,\dots,F)$ in Remark \ref{remidedia} and with \eqref{equopecomwis}). Thus (for instance) for a random element $(Y_{1:n},Z_{1:n})$ of $\R^{n}\times S_{1:n}^{\otimes}$ and $h_{1:n}\in \calL_{S_{1:n}}^{\otimes}$,
\begin{align}
||Y_{1:n}||_{\PP,\infty}=(||Y_{1}||_{\PP,\infty}, \dots, ||Y_{n}||_{\PP,\infty}),\\
\Esp{h_{1:n}(Z_{1:n})}\eqdef (\Esp{h_{1}(Z_{1})},\dots,\Esp{h_{n}(Z_{n})})
\end{align}
whenever the expectations are well defined. In order to avoid confusions, 
we will always introduce the dimension ``$n$'' for functionals  defined on sequences $a_{1:n}$ rather than on each one of its elements. See for instance \eqref{equdeflpnor}.

 \medskip

\noindent{\itshape\bfseries  $\ell^{p}$ norms.} We will use the notation
\begin{align}
\label{equdeflpnor}
|a_{1:m}|_{m,p}=\left(\sum_{k=1}^{m}|a_{k}|^{p}\right)^{1/p}
\end{align} 
($p\geq 1$) for the $\ell^{p}$ norm of $a_{1:m}\in [-\infty,\infty]^{m}$,  with the obvious convention when $a_{k}=\pm\infty$ for some $k$. 
%\end{enumerate}

%\begin{enumerate}[resume]
% \subsubsection*{Convex  hulls%, convex prism
% }

\medskip

\noindent{\itshape\bfseries Convex hulls.} In what follows,
\begin{align}
a_{1:n}^{1:m}=(a_{1:n}^{1},\dots,a_{1:n}^{m})
\end{align}
denotes an $m-$tuple whose elements are $n-$tuples (this can be thought of as a $m\times n$ matrix when convenient), and for consistency we will use the notation $t^{1:m}$ for any $m-$tuple that operates against $a_{1:n}^{1:m}$.

 Given $\cH_{1:n}\subset\calL_{S_{1:n}}^{\otimes}$ we will  denote by 
\begin{align}
\label{equdefconhul}
co(\cH_{1:n})\eqdef \bigcup_{m\in \N}\left\{t^{1:m}\cdot h_{1:n}^{1:m} \,|\, (h_{1:n}^{1:m},t^{1:m})\in \cH_{1:n}^{m}\times \in [0,1]^{m}, |t^{1:m}|_{m,1}=1 \right\}
\end{align}
the {\it convex hull of  $\cH_{1:n}$}, and by 
\begin{align}
\label{equdefconpri}
cobal(\cH_{1:n})\eqdef \bigcup_{m\in \N}\left\{t^{1:m}\cdot h_{1:n}^{1:m} \,|\, h_{1:n}^{1:m}\in \cH_{1:n}^{m},|t^{1:m}|_{m,1}\leq1 \right\}
\end{align}
the {\it convex balanced hull of $\cH_{1:n}$}. 

%\end{enumerate}
\brem
\label{remchacp}
It is easy to see from \eqref{equdefconpri} that 
$cobal(\cH_{1:n})$ is equal to $co(\{0\}\cup\cH_{1:n}\cup(-\cH_{1:n}))$, where ``$0$'' denotes the sequential function $f_{1:n}(z_{1:n})=(0,\dots,0)$.
  Since  $0\in co(\cH_{1:n}\cup(-\cH_{1:n}))$ we deduce that, in fact
\begin{align}
\label{equcpvscoequ}
cobal( \cH_{1:n})=co(\cH_{1:n}\cup(-\cH_{1:n})).
\end{align} 
In particular, $cobal( \cH_{1:n})=co(\cH_{1:n})$ if $\cH_{1:n}=-\cH_{1:n}$.
\erems

\noindent{\itshape\bfseries Pointwise measurability. } We will assume that all the sequential families $\cH_{1:n}$ considered here are  {\it pointwise measurable}, meaning that there exists a family 
\begin{align}
\label{fampoimes}
(h_{1:n}^{(l)})_{l}\subset \cH_{1:n}
\end{align}
 such that, for every given $h_{1:n}\in \cH_{1:n}$, there exists a sequence $(l_{m})_{m}$ with 
\begin{align}
h_{1:n}(z_{1:n})=\lim_{m}h_{1:n}^{(l_{m})}(z_{1:n})
\end{align}
for all $z_{1:n}$. This implies in particular that, for all $z_{1:n}\in S_{1:n}^{\otimes}$ and every continuous $\varphi_{n}: \R^{n}\to \R$,
\begin{align}
\sup_{h_{1:n}\in \cH_{1:n}}\varphi_{n}(h_{1:n}(z_{1:n}))=\sup_{l\in \N}\varphi_{n}(h_{1:n}^{(l)}(z_{1:n})).
\end{align}
for all $z_{1:n}\in S_{1:n}^{\otimes}$. 

%\end{enumerate}
%\subsubsection*{Sequences indexed by arbitrary sets}

\medskip

\noindent{\itshape\bfseries Sequences indexed by arbitrary sets.} Let us finally establish that all the definitions above can and will be extended to cases where, instead of $1,\dots,n$, the sequences are indexed by other sets, so if (for instance) $J\subset \N$ is any given set, $z_{J}$ denotes an element of the form $(z_{j})_{j\in J}$. We will also be careful to preserve a consistent notation, so if ${J}_{1}$ and $J_{2}$ are given, we will use the notation $z_{J_{1}}$ and $z_{J_{2}}$ only if these touples coincide on the indexes in $J_{1}\cap J_{2}$ (and therefore the restriction $z_{J_{1}\cap J_{2}}$ is unambiguously defined).

\subsection{A general remark on dependent samples}
\label{secremdep}

We will extend the results in sections \ref{p:tools} and \ref{secgenvc} that are obtained under the hypothesis of independent sampling  to the dependent case in a manner that is useful  when certain $\beta-$mixing coefficients decay rapidly enough. The beta-mixing coefficients are defined as follows:
\begin{dfn}[$\beta-$mixing coefficients]
\label{defbetmixcoe}
Let $\cA_{1}$ and $\cA_{2}$ be two sub-sigma algebras  of $\cA$. The $\beta-$mixing coefficient $\beta(\cA_{1},\cA_{2})$ between $\cA_{1}$ and $\cA_{2}$ is defined as
\begin{align}
\label{betmixdef}
\beta(\cA_{1},\cA_{2}):=&\Esp{\esssup_{A_{1}\in\cA_{1}}|\PP({A}_{1})-\PPro{A_{1}|\cA_{2}}|}
\end{align}
where $\esssup_{A_{1}\in \cA_{1}}$ denotes the essential supremum indexed by the element of $\cA_{1}$. If $Z_{1}, Z_{2}$ are random elements,
\begin{align}
\label{equdefbetranele}
\beta(Z_{1},Z_{2}):=\beta(\sigma(Z_{1}),\sigma(Z_{2})).
\end{align}
\end{dfn}

{See \cite[Proposition VI-1-1]{nevdispar} for a definition of the essential supremum,} from where it follows in particular that there exist a countable family $\{A_{1,n}\}_{n}\subset \cA_{1}$ such that,
\begin{align}
\beta(\cA_{1},\cA_{2})=\Esp{\sup_{n}|\PPro{A_{1,n}}-\PPro{A_{1,n}|\cA_{2}}|}
\end{align} 
\brem
\label{remprobetmix} 
If $\cA_{2}$ is countably generated, then
\begin{align}
\beta(\cA_{1},\cA_{2})&=\frac{1}{2}\sup_{(P_{{1}},P_{{2}})\in \cP_{\cA_{1}}\times \cP_{\cA_{2}}}\sum_{(A_{1},A_{2})\in{P_{{1}}}\times {P_{{2}}}}|\PPro{A_{1}\cap A_{2}}-\PP(A_{1})\PP(A_{2})|,
\end{align} 
where $\cP_{\cA_{k}}$ ($k=1,2$) denotes the family of finite partitions of $\Omega$ by $\cA_{k}-$sets.\footnote{This can be seen for instance by noticing that there exist increasing families of finite fields $\{\cA_{j,k}\}_{k}$ ($j=1,2$) with $\cup_{k}\cA_{j,k}\subset \cA_{j}$  such that 
\begin{align}
\beta(\cA_{1},\cA_{2})=\lim_{k}\lim_{l}\beta(\cA_{1,l},\cA_{2,k}),
\end{align}
and using elementary considerations on $\beta(\cA_{1},\cA_{2})$ when $\cA_{j}$ are finite fields. For a proof under slightly more restrictive hypotheses, see \cite[Proposition F.2.8]{douprimousou}.} . This representation holds in particular if $\cA_{k}:=\sigma(Z_{k})$  ($k=1,2$) as in \eqref{equdefbetranele}.
\erems
With this notion, we define the {\it past-to-present $\beta-$mixing coefficient of $m-$dependence} as follows:
 \begin{dfn}
 Let $Z_{J}$ be a random element of $S_{J}^{\otimes}$ ($J\subset \mathbb{Z}$). For every $m\in \N$, the $\beta-$coefficient of $m-$dependence of $Z_{J}$ is defined by
 \begin{align}
\beta_{Z_{J}}(m):=\sup_{k}\beta(Z_{J\cap(-\infty,k-m]},Z_{J\cap\{k\}}).
\end{align} 
 \end{dfn}

 \bl
 
Let $m\leq  n$ and define, for every $0\leq k<m $
\begin{align}
J_{m,k}=\{k+lm:l\in \mathbb{Z}\}\cap\{1,\dots,n\} 
\end{align}
If $Z_{1:n}$ is a random element of $S_{1:n}^{\otimes}$, then for  any $t\in \R$, any $a_{1:n}\in \R^{n}$, and any  $\cH_{1:n}\subset \calL_{S_{1:n}}^{\otimes}$, the inequality
\begin{align}
&\PPro{\sup_{h_{1:n}\in \cH_{1:n}}\{a_{1:n}\cdot h_{1:n}(Z_{1:n})\}> mt}\notag\\\leq&\sum_{k=0}^{m-1}\PPro{\sup_{h_{J_{m,k}}\in \cH_{J_{m,k}}}a_{J_{m,k}}\cdot h_{J_{m,k}}(Z_{J_{m,k}}^{*})
>t}+n\beta_{Z_{1:n}}(m)
\label{equdevbetmix}
\end{align}
holds, where $Z_{1:n}^{*}$ is an independent sequence with the same marginals as $Z_{1:n}$.
\el
\begin{proof}
This follows from Berbee's lemma \cite[Theorem 16.12]{braintstrmixvol2}. See \cite[Theorem 2.11 and Proposition 2.14]{bargob19seq} for a detailed proof.
\end{proof}

\medskip

The inequality \eqref{equdevbetmix} permits to extend conveniently the estimates on independent samples that will appear below to cases in which the corresponding sampling sequences satisfy $\beta_{Z_{1:n}}(m_{n})\to_{n} 0$ with an appropriate rate of decay for appropriate $m_{n}\leq n$. To illustrate with a classical case, assume that $Z_{1:n}$ is the finite-dimensional projection of a sequence $Z_{1:\infty}$ such that the {\it exponential mixing rate}
\begin{align}
\label{equexpmixrat}
\beta_{Z_{1:n}}(m)\leq r^{-m}, &\qquad \mbox{for some $r>1$ independent of $n$}
\end{align}
is verified\footnote{This is for instance implied by the condition of ``geometric ergodicity'' if the sample sequence $Z_{1:\infty}$ is a (strictly) stationary  Markov chain, see \cite[Theorems 21.13 and 21.19]{braintstrmixvol2} and references therein. This is also the case for causal ARMA processes with absolutely continuous i.i.d. innovations, as shown in \cite{mok}.}. Then given $\delta\in (nr^{-n},1)$ the choice $m=\lceil \log_{r}(n/\delta)\rceil$ gives 
\begin{align}
\beta_{Z_{1:n}}(m)\leq {\delta}/{n},
\end{align}
which combined with \eqref{equexpmixrat} gives a deviation inequality  that typically differs from the one for the independent case ``only'' by (essentially) a logarithmic factor. 

To be more concrete, if we know a bound of the type
 \begin{align}
 \label{equestind}
 \PPro{\sup_{h_{J}\in \cH_{J}}a_{J}\cdot h_{J}(Z_{J}^{*})
 >t}\leq F(t,|J|) 
 \end{align}
 giving uniform deviation bounds for sub-samples of size $|J|$ of an independent sequence $Z_{1:\infty}^{*}$ with the same marginals of $Z_{1:\infty}$, then under \eqref{equexpmixrat} the left-hand side of \eqref{equdevbetmix} is  upper bounded (ignoring divisibility issues) by 
 \begin{align}
 \label{equuppboufandlog}
 \log_{r}(n/\delta)F(t/\log_{r}(n/\delta), n/(\log_{r}(n/\delta)))+\delta.
 \end{align}

 For further illustration, we will apply this idea in  Remarks \ref{rembetmixrad} and \ref{remdepbouvc} below,  {assuming that} $n$ is divisible by $ \log_{r} (n/\delta)$. The reader is invited to write down the estimate for general $\delta$ and to perform analogous estimations for the {\it subpolynomial mixing case} in which 
 \begin{align}
 \beta_{Z_{1:n}}(m)\leq m^{-r},&\qquad\mbox{for some $r>1$.}
 \end{align}

 \section{Generalization bounds via the   Rademacher theory
\label{p:tools}}

{This section presents some of the main ideas within the so-called ``Rademacher theory'', interpreted in terms of confidence intervals for regression schemes with general loss functions. We begin in Section \ref{secradcom} by presenting the notion of Rademacher complexity and some of its properties; then, in Section \ref{seclardevrad}, we  present some deviation inequalities involving the Rademacher complexity which are relatively straightforward applications of McDiarmid's inequality. Section \ref{secmaslem} presents the classical relation, known as ``Massart's lemma'', between the Rademacher complexity and the geometric notion of covering numbers, including some applications. Finally, in Section \ref{secentest}, we introduce the notion of {\it entropy estimates}, which allow us to illustrate via some relevant examples the way in which the estimates previously developed apply to particular sets of hypotheses.
 
 \subsection{Rademacher complexity}
 \label{secradcom}
 We begin by reminding the following:
 
 \begin{dfn}
A Rademacher sequence of lenght $n$ is an i.i.d. sequence $U_{1:n}$ with
\begin{align}
\PPro{U_{1}=1}=\PPro{U_{1}=-1}=\frac{1}{2}.\fe
\end{align}
 \end{dfn}
 
 The essential notion for what follows is that of {\it Rademacher complexity} (for families of sequential functions), defined as:
 
 \begin{dfn}
 \label{defradcom}
 The empirical Rademacher complexity of a   family of sequential  functions $\cH_{1:n}\subset \calL_{S_{1:n}}^{\otimes}
$ at    $z_{1:n}$ is defined as 
\begin{align}
\label{equdefradcomemp}
\Radcemp(\cH_{1:n},z_{1:n})\eqd \Esp{\sup_{h_{1:n}\in \cH_{1:n}}(U_{1:n}\cdot h_{1:n}(z_{1:n}))}
\end{align}
where $U_{1:n}$ a Rademecher sequence. 
If $Z_{1:n}$ is a random element of $S_{1:n}^{\otimes}$, the Rademacher complexity of $\cH_{1:n}$ with respect to $Z_{1:n}$ is defined as
\begin{align}
\label{equdefradcomave}
\Radcave(\cH_{1:n},Z_{1:n})
\eqd \Esp{\Radcemp(\cH_{1:n},Z_{1:n})}=
\Esp{\sup_{h_{1:n}\in \cH_{1:n}}(U_{1:n}\cdot h_{1:n}(Z_{1:n})},
\end{align}
where $U_{1:n}$ is a Rademacher sequence independent of $Z_{1:n}$. 

If $S_{1:n}=S^{n}$ and $\cH\subset \calL(S)$ is given, we define
\begin{align}
\label{equdefradcomclafam}
\Radcemp(\cH,z_{1:n})\eqdef \Radcemp(diag(\cH)_{1:n},z_{1:n}),&\qquad \Radcave(\cH,Z_{1:n})\eqdef \Radcave(diag(\cH)_{1:n},Z_{1:n}),
\end{align}
where $diag(\cH)_{1:n}$ is defined by \eqref{equdefdiafam}.\fe
\end{dfn}
\brem
Notice that in this definition $\Radcemp$ can be interpreted as a particular instance of $\Radcave$ (consider a point measure at $z_{1:n}$). We keep both definitions separate in order to eventually use the inequality
\begin{align}
\label{equuppbourad}
\Radcave(\cH_{1:n},Z_{1:n})\leq \sup_{z_{1:n}}\Radcemp(\cH_{1:n},z_{1:n}) 
\end{align}
where the $\sup$ is taken over a set $S\subset S_{1:n}^{\otimes}$ supporting $Z_{1:n}$ ($\PP(Z_{1:n}\in S)=1$). This permits to upper estimate $\Radcave(\cH_{1:n},Z_{1:n})$ uniformly over all distributions supported on $S$.\fe
\erems

\bp
\label{proproradcom}
With the notation  \eqref{equdefdirsum}, \eqref{equdefconpri}, the inequalities 
\begin{align}
\label{equproradcom}
\Radcemp(\cH_{1:n}^{(1)}+\cH_{1:n}^{(2)},z_{1:n})=&\Radcemp(\cH_{1:n}^{(1)},z_{1:n})+\Radcemp(\cH_{1:n}^{(2)},z_{1:n})\\
\Radcemp(co(\cH_{1:n}),z_{1:n})=&\Radcemp(\cH_{1:n},z_{1:n})
\label{equprodadcomconco}\\
\Radcemp(cobal( \cH_{1:n}),z_{1:n})\leq & 2 \Radcemp(\cH_{1:n},z_{1:n})
\label{equprodadcomcon}
\end{align} 
hold for every $z_{1:n}\in S_{1:n}^{\otimes}$ and  $\cH_{1:n}, \cH_{1:n}^{(k)}\subset \calL_{S_{1:n}}^{\otimes}$ ($k=1,2$). The same inequalities hold (by integration) when $\Radcemp(\cdot,z_{1:n})$ is replaced by $\Radcave(\cdot,Z_{1:n})$, where $Z_{1:n}$ is a random element of $S_{1:n}$, provided that these complexities are finite.  
 \ep
 
 \brem
 The inequality \eqref{equprodadcomcon} is tight (but see Remark \ref{remequradcob}): if $n=2$ and  $\cH_{1:2}=\{(1,0),(0,1)\}$ where $(a,b)$ denotes the constant sequential functions $(h_{1}(z_{1}),h_{2}(z_{2}))=(a,b)$, then for every $(z_{1},z_{2})$
 \begin{align}
 \Radcemp(cobal( \cH_{1:2}),z_{1:2})=1=2\,\Radcemp( \cH_{1:2},z_{1:2}).
 \end{align}
 \erems
 
 \begin{proof}(of Proposition \ref{proproradcom})
 The equality \eqref{equproradcom} follows from the equality
 \begin{align}
 \label{equproequsupas}
 &\sup_{(h_{1:n}^{(1)},h_{1:n}^{(2)})\in \cH_{1:n}^{(1)}\times\cH_{1:n}^{(2)}}u_{1:n}\cdot(h_{1:n}^{(1)}(z_{1:n})+h_{1:n}^{(2)}(z_{1:n}))\notag\\
= &\sup_{h_{1:n}^{(1)}\in \cH_{1:n}^{(1)}}u_{1:n}\cdot h_{1:n}^{(1)}(z_{1:n})+\sup_{h_{1:n}^{(2)}\in \cH_{1:n}^{(2)}}u_{1:n}\cdot h_{1:n}^{(2)}(z_{1:n}),
 \end{align}
 valid for every $u_{1:n}\in \{-1,1\}^{n}$. The proof of ``$\leq$'' in \eqref{equproequsupas}   is obvious,
 whereas ``$\geq$'' follows from the following observation:  given  $\epsilon>0$ and $u_{1:n}\in\{-1,1\}^{n}$, there exist  $(h_{1:n}^{(1,u_{1:n})},h_{1:n}^{(2,u_{1:n})})\in \cH_{1:n}^{(1)}\times \cH_{1:n}^{(2)}$ such that 
 \begin{align}
& \sup_{h_{1:n}^{(1)}\in \cH_{1:n}^{(1)}}u_{1:n}\cdot h_{1:n}^{(1)}(z_{1:n})+\sup_{h_{1:n}^{(2)}\in \cH_{1:n}^{(2)}}u_{1:n}\cdot h_{1:n}^{(2)}(z_{1:n})\notag\\
\leq& u_{1:n}\cdot(h_{1:n}^{(1,u_{1:n})}(z_{1:n})+h_{1:n}^{(2,u_{1:n})}(z_{1:n}))+\epsilon\notag\\
\leq&\sup_{(h_{1:n}^{(1)},h_{1:n}^{(2)})\in \cH_{1:n}^{(1)}\times\cH_{1:n}^{(2)}}u_{1:n}\cdot(h_{1:n}^{(1)}(z_{1:n})+h_{1:n}^{(2)}(z_{1:n}))+\epsilon
\label{equproannine}
 \end{align}
 which gives the conclusion by  letting $\epsilon\to 0$.
 
 The proof of \eqref{equprodadcomconco} is left to the reader. 
 
 To prove \eqref{equprodadcomcon} notice that, by \eqref{equcpvscoequ} and \eqref{equprodadcomconco}, if $U_{1:n}$ is a Rademacher sequence,
 \begin{align}
\Radcemp(cobal( \cH_{1:n}),z_{1:n})=&\Radcemp(\cH_{1:n}\cup(-\cH_{1:n}),z_{1:n})
\notag\\
=&\Esp{\max\{\sup_{h_{1:n}\in \cH_{1:n}}U_{1:n}\cdot h_{1:n}(z_{1:n}),\sup_{h_{1:n}\in \cH_{1:n}}(-U_{1:n})\cdot h_{1:n}(z_{1:n})\}}
\notag\\
\leq&\Esp{\sup_{h_{1:n}\in \cH_{1:n}}U_{1:n}\cdot h_{1:n}(z_{1:n})}+\Esp{\sup_{h_{1:n}\in \cH_{1:n}}(-U_{1:n})\cdot h_{1:n}(z_{1:n})}\notag\\
=& 2 \Radcemp(\cH_{1:n},z_{1:n}),
 \end{align}
 where we used the fact that $-U_{1:n}$ is a Rademacher sequence.\fe

 \end{proof}
 \brem
 \label{remequradcob}
 It is clear from the proof of \eqref{equprodadcomcon} that if for every $h_{1:n}\in \cH_{1:n}$ there exists $h_{1:n}'\in \cH_{1:n}$ with 
 \begin{align}
 h_{1:n}(z_{1:n})=-h_{1:n}'(z_{1:n})
 \end{align}
  then
 \begin{align}
 \label{equradcomcpeqrancomspe}
\Radcemp(cobal( \cH_{1:n}),z_{1:n})=\Radcemp(\cH_{1:n},z_{1:n}).\fe
 \end{align}
 \erems
 
\brem
\label{remuppboufunbetcon}
Clearly, \eqref{equprodadcomconco} implies that for any $\cH_{1:n}'\subset\cH_{1:n}\subset co(\cH_{1:n}')$
\begin{align}
\Radcemp(\cH_{1:n},z_{1:n})=&\Radcemp(\cH_{1:n}',z_{1:n}),
\end{align}
with the respective analogous consequence for $\Radcave(\cH_{1:n},Z_{1:n})$.
\erem
 
\subsection{Large deviations via Rademacher estimates}
\label{seclardevrad}

The following lemma, which we prove for the sake of completeness, is a typical tool in the nonasymptotic analysis of empirical minimization.
\bl
\label{lemsym}
If $Z_{1:n}$ has independent components, then for any fixed $u_{1:n}\in \{-1,1\}^{n}$
\begin{align}
\Esp{\sup_{h_{1:n}\in \cH_{1:n}}u_{1:n}\cdot(h_{1:n}(Z_{1:n})-\Esp{h_{1:n}(Z_{1:n})})
}\leq 2\Radcave(\cH_{1:n},Z_{1:n}).
\end{align}
\el
\proof 
 Pick a sequence $Z_{1:n}'$ of random variables with $Z_{1:n}'\sim Z_{1:n}$ and $Z_{1:n}'$ independent of $Z_{1:n}$. 
  Notice that for every $u_{1:n}' \in \{-1, 1\}^{n}$,
\begin{align}
\Esp{\sup_{h_{1:n}\in \cH_{1:n}}u_{1:n}'\cdot (h_{1:n}(Z_{1:n})-h_{1:n}(Z_{1:n}'))
}
=\Esp{\sup_{h_{1:n}\in \cH_{1:n}}{u}_{1:n}\cdot (h_{1:n}(Z_{1:n})-h_{1:n}(Z_{1:n}')
},
\label{equradconcon}
\end{align}
by the independence of $(Z_{1:n},Z_{1:n}')$ and because $Z_{1:n}\sim Z_{1:n}'$ (given $J\subset \{1,\dots,n\}$, exchange of signs in $h_{J}$  for all $h_{1:n}\in \cH_{1:n}$ does not affect the finite-dimensional distributions of $(h_{1:n}(Z_{1;n})-h_{1:n}(Z_{1:n}'))_{h
\in \cH}$). 
This implies that for every Rademacher sequence $U_{1:n}'$ independent of  $(Z_{1:n},Z_{1:n}')$, the equality
\begin{align}
\Esp{\sup_{h_{1:n}\in \cH_{1:n}}{U}_{1:n}'\cdot (h_{1:n}(Z_{1:n}')-h_{1:n}(Z_{1:n}))}=\Esp{\sup_{h_{1:n}\in \cH_{1:n}}{u}_{1:n}\cdot (h_{1:n}(Z_{1:n}')-h_{1:n}(Z_{1:n})}
\end{align} 
holds (condition on $U_{1:n}'$ and use \eqref{equradconcon}).

 It follows by monotonicity of the (conditional) expectation  that 
\begin{align}
&\Esp{\sup_{h_{1:n}\in \cH_{1:n}}u_{1:n}\cdot(h_{1:n}(Z_{1:n})-\Esp{h_{1:n}(Z_{1:n})})}\notag\\
=&\Esp{\sup_{h_{1:n}\in \cH_{1:n}}\Esp{u_{1:n}\cdot({h_{1:n}(Z_{1:n})}-h_{1:n}(Z_{1:n}'))
|Z_{1:n}}}
\notag\\
\leq & \Esp{\sup_{h_{1:n}\in \cH_{1:n}}
u_{1:n}\cdot({h_{1:n}(Z_{1:n})}-h_{1:n}(Z_{1:n}'))
}
\notag\\
=&\Esp{\sup_{h_{1:n}\in \cH_{1:n}}U_{1:n}'\cdot({h_{1:n}(Z_{1:n})}-h_{1:n}(Z_{1:n}'))
}\notag\\
\leq &\Esp{\sup_{h_{1:n}\in \cH_{1:n}}U_{1:n}'\cdot{h_{1:n}(Z_{1:n})}}+\Esp{\sup_{h_{1:n}\in \cH_{1:n}}(-U_{1:n}')\cdot h_{1:n}(Z_{1:n}')
}
\notag\\
=&2\,\Radcave(\cH_{1:n},Z_{1;n}).~\finproof
\end{align}
   
We also remind the following version of  {McDiarmid}'s equality \cite[Theorem A.2]{gyor:kohl:krzy:walk:02}. \cite[Lemma A.4]{gyor:kohl:krzy:walk:02})
\bl
\label{lemazuhoemcdia}
If $Z_{1:n}$ is a random element of $S_{1:n}^{\otimes}$ with independent components and if  $K:S_{1:n}^{\otimes}\to \R$ is a Borel measurable function with 
\begin{align}
\label{equconazumcdia}
|\Esp{K(Z_{1:k-1},z_{k},Z_{k+1:n})|Z_{1:k-1}}-\Esp{K(Z_{1:k-1},z_{k}',Z_{k+1:n})|Z_{1:k-1}}|\leq c_{k}
\end{align}
$\PP$-a.s., for $\PP_{Z_{k}}\times \PP_{Z_{k}}$-almost every $(z_{k},z_{k}')$ and for some  $c_{1:n}\in [0,\infty]^{n}$ then, for $u=\pm 1$ and every $\epsilon>0$, 
\begin{align}
\PP(u(K(Z_{1:n})-\Esp{K(Z_{1:n})})>\epsilon)\leq \exp\left(-2({\epsilon}/{|c_{1:n}|_{n,2}})^{2}\right).\fe
\end{align}
\el

This has the following straightforward consequence:
\begin{thm}
\label{corunidevsupaveind}
 Let $\cH_{1:n}\subset \calL_{S_{1:n}}^{\otimes}$,  let $Z_{1:n}$ be a random element of $S_{1:n}^{\otimes}$ with independent components, 
  and let $H_{1:n}:S_{1:n}^{\otimes}\to [0,\infty]^{n}$ be defined as in \eqref{equdefseqfuncaph}.
 Then for every  $(u_{1:n},\epsilon)\in \{-1,1\}^{n}\times (0,\infty)$,

\begin{align}
&\PP\left({\sup_{h_{1:n}\in \cH_{1:n}}
u_{1:n}\cdot(h_{1:n}(Z_{k})-\Esp{h_{1:n}(Z_{1:n})})} >\epsilon+ 2\Radcave(\cH_{1:n},Z_{1:n}) \right)\notag\\
\leq &\exp\left(-(\epsilon/\sqrt{2} ||\,|H_{1:n}(Z_{1:n})|_{n,2}||_{\PP,\infty})^{2}
\right).
\label{equdevsupind}
\end{align}
\end{thm}
\proof 
Take 
\begin{align}
K(z_{1:n})\eqd \sup_{h_{1:n}\in \cH_{1:n}}u_{1:n}\cdot(h_{1:n}(z_{1:n})-\Esp{h_{1:n}(Z_{1:n})}).
\end{align}
In this case 
\begin{align}
\label{equbousupinfk}
|K(Z_{1:k-1},z_{k},Z_{k+1:n})-K(Z_{1:k-1},z_{k}',Z_{k+1:n})|\leq 2||H_{k}(Z_{k})||_{\PP,\infty}
\end{align}
 for $\PP_{Z_{k}}\times\PP_{Z_{k}}-$a.e. $(z_{k},z_{k}')$). Combining Lemma \ref{lemsym}  with Lemma \ref{lemazuhoemcdia} we obtain
\begin{align}
\PPro{K(Z_{1:n})>\epsilon+ 2\Radcave(\cH_{1:n},Z_{1:n}) }{\leq} & \PPro{K(Z_{1:n})>\epsilon+ \Esp{K(Z_{1:n})} }\notag\\
{\leq}& \exp\left(-(\epsilon/\sqrt{2}||\,|H_{1:n}(Z_{1:n})|_{n,2}||_{\PP,\infty})^{2}
\right),
\label{equappmcdiasup}
\end{align}
where the first inequality follows from Lemma \ref{lemsym} and the second follows from Lemma \ref{lemazuhoemcdia}, \eqref{equbousupinfk} and \eqref{equequl2norlinfnorlinfnorl2nor}.\fe  

\brem\footnote{This remark is a transcription of a comment by H-D Nguyen.}
\label{remhoa}
If $\cH_{1:n}\subset \calL_{S_{1:n}}^{\otimes}$ is a sequential family of nonnegative functions, in the sense that
\begin{align}
\label{equhypposlosfun}
 \mbox{\it for every $h_{1:n}\in \cH_{1:n}$ and every $1\leq k\leq n$, $h_{k}(Z_{k})\geq 0$, $\PP-$a.s.}
 \end{align}
  then we the following refinement of \eqref{equdevsupind} holds:
  \begin{align}
&\PP\left({\sup_{h_{1:n}\in \cH_{1:n}}
u_{1:n}\cdot(h_{1:n}(Z_{k})-\Esp{h_{1:n}(Z_{1:n})})} >\epsilon+ 2\Radcave(\cH_{1:n},Z_{1:n}) \right)\notag\\
\leq &\exp\left(-2(\epsilon/ ||\,|H_{1:n}(Z_{1:n})|_{n,2}||_{\PP,\infty})^{2}
\right),
\label{equdevsupindref}
\end{align}
as can easily be seen by an examination of the argument leading to \eqref{equbousupinfk} (which allows to get rid of the factor $2$ under \eqref{equhypposlosfun}) and the steps underneath it: an equivalent formulation of this is that, under \eqref{equhypposlosfun}, we can replace $\epsilon$ by $\epsilon/2$ at the left-hand side of \eqref{equdevsupind}.
\erems
\medskip

We introduce now,  for notational convenience, the functionals defined for $(u_{1:n},h_{1:n},z_{1:n})\in\{-1,1\}^{n}\times \cH_{1:n}\times S_{1:n}^{\otimes}$ by
\begin{align}
\label{funauxcrefratprosinh}
w_{h_{1:n}}(z_{1:n})=&u_{1:n}\cdot (\Esp{h_{1:n}(Z_{1:n}')}-h_{1:n}(z_{1:n}))\\
w(z_{1:n})=&\sup_{h_{1:n}\in \cH_{1:n}}w_{h_{1:n}}(z_{1:n})\\
\label{funauxcrefratprok}
k_{h_{1:n}}(z_{1:n})=&\min\{k\in \{0,\dots, n\}: w_{h_{1:n}}(z_{1:n})\leq kw(z_{1:n})/n\}.
\end{align}
where $Z_{1:n}'$ is an independent copy of $Z_{1:n}$\footnote{This ``ghost sample'' has a purely auxiliary value here: one can avoid it by changing some of the expectations below by integrals with respect to the law of $Z_{1:n}$.}. Notice in particular that, if $w_{h_{1:n}}(z_{1:n})\leq 0$,  then $k_{h_{1:n}}(z_{1:n})=0$. In the case in which $w_{h_{1:n}}(z_{1:n})> 0$, $k_{h_{1:n}}(z_{1:n})$ is just the integer part of $n (w_{h_{1:n}}(z_{1:n})/w(z_{1:n}))$. 

\begin{thm}
\label{thegenuppbou}
Under the hypotheses of Theorem \ref{corunidevsupaveind}, let
 \begin{align}
 \label{empandglosolgen}
 \tilde{h}_{1:n}\in \arg\min_{h_{1:n}\in \cH_{1:n}}
u_{1:n}\cdot\Esp{h_{1:n}(Z_{1:n})},&\qquad\hat{h}_{1:n}=\hat{h}_{1:n,Z_{1:n}}\in \arg\min_{h_{1:n}\in \cH_{1:n}}
u_{1:n}\cdot h_{1:n}(Z_{1:n}).
\end{align}

Then for any $(\epsilon,\eta,k)\in[0,\infty)\times[0,\infty)\times\{0,\dots,n\}$,  and with the notation \eqref{funauxcrefratprosinh}- \eqref{funauxcrefratprok}, the bound
\begin{align}
&\PPro{u_{1:n}\cdot\Esp{(\hat{h}_{1:n}(Z_{1:n}')-\tilde{h}_{1:n}(Z_{1:n}'))|Z_{1:n}} \geq \epsilon+\eta+
2\Radcave(\cH_{1:n},Z_{1:n})k/n \big| k_{\hat{h}_{1:n}}(Z_{1:n})=k}\notag\\
\leq&\exp\left(-(\epsilon n/\sqrt{2} k ||\,|H_{1:n}(Z_{1:n})|_{n,2}||_{\PP,\infty})^{2}
\right)\un_{\{k>0\}}+\PPro{-w_{\tilde{h}_{1:n}}(Z_{1:n})\geq\eta}.
\label{equboutheupebougen}
\end{align}
holds.

\brem
By considering the case $\cH_{1:n}=\{\tilde{h}_{1:n}\}$ and exchanging the signs in $u_{1:n}$, the bound
\begin{align}
\label{equuppboudevhtil}
\PPro{-w_{\tilde{h}_{1:n}}(Z_{1:n})\geq\eta}\leq \exp(-(\eta/(\sqrt{2}||\,|\tilde{h}_{1:n}(Z_{1:n})|_{n,2}||_{\PP,\infty}))^{2})
\end{align} 
follows from \eqref{equdevsupind}. The bound \eqref{equuppboudevhtil} may nonetheless be suboptimal, and alternatives to \eqref{equdevsupind} (like Bernstein's inequality \cite[Lemma A.2]{gyor:kohl:krzy:walk:02}) may in some cases give a better upper bound. 
\erems
\brem
The upper bound \eqref{equboutheupebougen} should be compared with the more classically flavoured upper bound \eqref{equdeverrlosradcom} presented below, which follows from \eqref{equboutheupebougen}. It is clear that \eqref{equboutheupebougen} it is strictly tighter than \eqref{equdeverrlosradcom} (consider for instance the case in which $\tilde{h}_{1:n}$ is constant).
\erems
\end{thm}
\begin{proof} (of Theorem \ref{thegenuppbou})
By the definition of $\hat{h}_{1:n}, \tilde{h}_{1:n}$, $Z_{1:n}'$, and the functionals \eqref{funauxcrefratprosinh}- \eqref{funauxcrefratprok}, we get

\begin{align}
 u_{1:n}\cdot\Esp{(\hat{h}_{1:n}(Z_{1:n}')-\tilde{h}_{1:n}(Z_{1:n}'))|Z_{1:n}} 
=&w_{\hat{h}_{1:n}}(Z_{1:n})+u_{1:n}\cdot\left(\hat{h}_{1:n}(Z_{1:n})-\Esp{\tilde{h}_{1:n}(Z_{1:n}')}\right)
\notag\\
\leq&w_{\hat{h}_{1:n}}(Z_{1:n})-w_{\tilde{h}_{1:n}}(Z_{1:n})
\notag\\
\leq &{k_{\hat{h}_{1:n}}(Z_{1:n})w(Z_{1:n})}/{n}-w_{\tilde{h}_{1:n}}(Z_{1:n}).
\notag\\
\end{align}
From where it follows via the union bound and \eqref{equdevsupind} that for every $(\epsilon,\eta,k)\in [0,\infty)\times[0,\infty)\times\{0,\dots,n\}$
\begin{align}
&\PPro{a_{1:n}\cdot\Esp{(\hat{h}_{1:n}(Z_{1:n}')-\tilde{h}_{1:n}(Z_{1:n}'))|Z_{1:n}}>\epsilon+\eta+2\Radcave(\cH_{1:n},Z_{1:n})k/n\big| k_{\hat{h}_{1:n}}(Z_{1:n})=k}
\notag\\
\leq &\PPro{k_{\hat{h}_{1:n}}(Z_{1:n})w(Z_{1:n})/{n}>\epsilon +
2\Radcave(\cH_{1:n},Z_{1:n})k/n\big|k_{\hat{h}_{1:n}}(Z_{1:n})=k} +\PPro{-w(\tilde{h}_{1:n},Z_{1:n})>\eta}\notag\\
= &\PPro{ w(Z_{1:n})>\epsilon n/k+
2\Radcave(\cH_{1:n},Z_{1:n})} \un_{\{k>0\}}+\PPro{-w(\tilde{h}_{1:n},Z_{1:n})>\eta}\notag\\
\leq&\exp\left(-(\epsilon n/(\sqrt{2} k ||\,|H_{1:n}(Z_{1:n})|_{n,2}||_{\PP,\infty}))^{2}
\right)\un_{\{k>0\}}+\PPro{-w(\tilde{h}_{1:n},Z_{1:n})>\eta}. \notag\\
&
\\
\notag \qed
\end{align}
\end{proof}

\bcor
\label{cortheradbou}
 Let $\cH_{1:n}\subset \calL_{S_{1:n}}^{\otimes}$,  let $Z_{1:n}$ be a random element of $S_{1:n}^{\otimes}$ with independent components, 
  let $H_{1:n}:S_{1:n}^{\otimes}\to [0,\infty]^{n}$ be defined as in \eqref{equdefseqfuncaph},   
 and let
 \begin{align}
 \label{empandglosolgenone}
 \tilde{h}_{1:n}\in \arg\min_{h_{1:n}\in \cH_{1:n}}
\sum_{k=1}^{n}\Esp{h_{k}(Z_{k})},&\qquad\hat{h}_{1:n}=\hat{h}_{1:n,Z_{1:n}}\in \arg\min_{h_{1:n}\in \cH_{1:n}}
\sum_{k=1}^{n}h_{k}(Z_{k}).
\end{align}
Then for any $\epsilon>0$ and any independent copy  $Z_{1:n}'$ of  $Z_{1:n}$
\begin{align}
\PP\left(\sum_{k=1}^{n}\Esp{(\hat{h}_{k}(Z_{k}')-\tilde{h}_{k}(Z_{k}'))|Z_{1:n}}>2(\epsilon + 
\Radcave(\cH_{1:n},Z_{1:n}))\right)\notag\\
\leq 2\exp\left(-\left({\epsilon}/{(\sqrt{2}||\,|H_{1:n}(Z_{1:n})|_{n,2}||_{\PP,\infty})}\right)^{2}\right).
\label{equdeverrlosradcom}
\end{align}
In particular, for any $\delta>0$, the inequality
\begin{align}
\sum_{k=1}^{n}\Esp{(\hat{h}_{k}(Z_{k}')-\tilde{h}_{k}(Z_{k}'))|Z_{1:n}}\leq 2\big(
{||\,|H_{1:n}(Z_{1:n})|_{n,2} \,||_{\PP,\infty}}\sqrt{2\log({2}/{\delta})}+ \Radcave(\cH_{1:n},Z_{1:n})
\big)
\label{equradestconint}
\end{align}
holds with probability at least $1-\delta$.
\ecor

\begin{proof}  Let us for convenience denote
\begin{align}
\label{equdeluppbou}
\delta(\epsilon,k):=\exp\left(-(\epsilon n/(\sqrt{2} k ||\,|H_{1:n}(Z_{1:n})|_{n,2}||_{\PP,\infty}))^{2}
\right)\un_{\{k>0\}}.
\end{align}
In the context of Theorem \ref{thegenuppbou}, considering the case in which $u_{1:n}:=(1,\dots,1)$, an application of \eqref{equboutheupebougen} gives that the left-hand side of \eqref{equdeverrlosradcom} is equal to
\begin{align}
&\PPro{u_{1:n}\cdot\Esp{\hat{h}_{1:n}(Z_{1:n}')-\tilde{h}(Z_{1:n}')|Z_{1:n}}\geq 2(\epsilon + 
\Radcave(\cH_{1:n},Z_{1:n}))}\notag\\
=&\sum_{k=0}^{n}\Esp{\PPro{u_{1:n}\cdot\Esp{\hat{h}_{1:n}(Z_{1:n}')-\tilde{h}(Z_{1:n}')|Z_{1:n}}\geq 2(\epsilon + 
\Radcave(\cH_{1:n},Z_{1:n}))|k_{\hat{h}_{1:n}}=k}\un_{\{k_{\hat{h}_{1:n}}=k\}}}\notag\\
\leq&\sum_{k=0}^{n}\delta(\epsilon,k)\PPro{k_{\hat{h}_{1:n}}=k}+\PPro{w_{\tilde{h}}(Z_{1:n})>\epsilon}
\leq2\delta(\epsilon,n).
\label{equdeverrlosradcomref}
\end{align}
where in the last inequality we used the fact that $\eqref{equdeluppbou}$ is increasing in $k$ to upper bound $\delta(\epsilon,k)$, and where we used the fact that  the right hand side of \eqref{equuppboudevhtil} is upper bounded by $\delta(\epsilon,n)$. Notice that \eqref{equdeverrlosradcom} is the inequality between the extremes of \eqref{equdeverrlosradcomref}. The upper bound \eqref{equradestconint} follows by equating $2\delta(\epsilon,n)$ to $\delta$ and solving for $\epsilon$.
\end{proof}

\brem
\label{rembetmixrad}
Applying the observations from Section \ref{secremdep} (see especially \eqref{equdevbetmix}) to the second term in \eqref{equappmcdiasup}, it is easy to show that if \eqref{equexpmixrat} holds, then  
\begin{align}
\PP\left(\frac{}{} \right.&{\sup_{h_{1:n}\in \cH_{1:n}}
u_{1:n}\cdot(h_{1:n}(Z_{k})-\Esp{h_{1:n}(Z_{1:n})})}\notag\\
&\left.>2\left\lceil\log_{r}(2n/\delta)\right\rceil(\epsilon+ \max_{1\leq k\leq \log_{r}(n/\delta)}
\Radcave(\cH_{J_{k}},Z_{J_{k}}^{*})) \right)\notag\\
\leq &2\left\lceil\log_{r}(2n/\delta)\right\rceil \exp\left(-\left({\epsilon}/{(\sqrt{2}\max_{1\leq k \leq \log_{r}(n/\delta)}||\,|H_{J_{k}}(Z_{J_{k}}^{*})|_{|J_{k}|,2}||_{\PP,\infty}})\right)^{2}\right)+\delta/2
\end{align}
where $J_{k}:=\{k+l\left\lceil\log_{r}(2n/\delta)\right\rceil\}_{l}\cap \{1,\dots,n\}$ and $Z_{1:n}^{*}$ is an independent sequence with the same marginals as $Z_{1:n}$. This carries further to the fact that, always under \eqref{equexpmixrat}, the conclusion of Corollary \ref{cortheradbou} holds if the right-hand side of \eqref{equradestconint} is replaced by
 \begin{align}
 &2^{3/2}\left\lceil\log_{r}(2n/\delta)\right\rceil\left(
{\max_{1\leq k \leq \log_{r}(2n/\delta)}||\,|H_{J_{k}}(Z_{J_{k}}^{*})|_{|J_{k}|,2}||_{\PP,\infty}}\sqrt{\log(4\left\lceil\log_{r}(2n/\delta)\right\rceil/\delta)}\right.\notag\\
+& \left.\max_{1\leq k\leq \log_{r}(2n/\delta)}\Radcave(\cH_{J_{k}},Z_{J_{k}}^{*})
\right).
\label{equradestconintbetmix}
\end{align}
(this estimate implicitly requires that $r^{-n}\leq \delta/2n\leq r^{-1}$).
\erems

\subsection{Massart's lemma and covering numbers}
\label{secmaslem}

Remember the definition of an $r-$covering in a semimetric space:
\begin{dfn}
\label{defcovnum}
Given a semimetric space $(\calL, d)$, $\cH\subset \calL$, and $r\in [0,\infty)$, an $r-$covering of $\cH$ with respect to $d$ is a set $\calL_{0}\subset\calL$ such that for every $h\in \cH$ there exists $l_{h}\in \calL_{0}$ with 
\begin{align}
d(l_{h},h)\leq r.
\end{align}
If $(\calL,||\cdot||)$ is a seminormed vector space, we define coverings with respect to $||\cdot||$ via the semimetric $d(l_{1},l_{2})=||l_{1}-l_{2}||$.
\end{dfn}

We proceed now to establish a relationship between the Rademacher complexities of a sequential family $\cH_{1:n}$ at $Z_{1:n}$ and the covering numbers associated to the empirical $L^{1}$-norm of $\cH_{1:n}$ at $Z_{1:n}$. 
\begin{dfn}
\label{defempconnum}
Consider the empirical $L^{1}-$norm of $\calL_{S_{1:n}}^{\otimes}$ at $z_{1:n}\in S_{1:n}^{\otimes}$, defined by
\begin{align}
|f_{1:n}|_{z_{1:n},1}\eqdef \frac{1}{n}\sum_{k=1}^{n}|f_{j}(z_{j})|
\end{align}
for every $f_{1:n}\in \calL_{S_{1:n}}^{\otimes}$.
Given a family of sequential functions $\cH_{1:n}\subset \calL_{S_{1:n}}^{\otimes}$ and $r\in[0,\infty)$, the $r$-convering number $\cN_{1}(\cH_{1:n},z_{1:n},r)$ is the smallest $m\in \N$ such that there exists an $r-$covering of size $m$ of $\cH_{1:n}$ with respect to $|\cdot|_{z_{1:n},1}$  (with the convention $\inf \emptyset=\infty$).
If $S_{1:n}=S^{n}$ and $\cH\subset \calL(S)$ is given, we define
\begin{align}
\cN_{1}(\cH,z_{1:n},r)\eqdef \cN_{1}(diag(\cH)_{1:n},z_{1:n},r),
\end{align}
where $diag(\cH)_{1:n}$ is as in \eqref{equdefdiafam}.\fe
\end{dfn}

The connection of covering numbers and Rademacher complexities is a consequence of Massart's lemma (see e.g. \cite[Theorem 3.7 page 35]{mohri2018foundations}).

\bl
\label{lemmaslem}
Given  $n\in \N$, a Rademacher sequence $U_{1:n}$, and a (finite) subset $\{a_{1:n}^{(k)}\}_{k=1}^{m}\subset \R^{n}$, the inequality
\begin{align}
\label{equbouespmaxfin}
\Esp{\max_{ k=1,\ldots,m}{(U_{1:n}\cdot a_{1:n}^{(k)}})}\leq \max_{k=1,\ldots,m}|a_{1:n}^{(k)}|_{n,2}\sqrt{2\log( m)}
\end{align}
holds.\fe
\el

Applying this result one shows easily that
\begin{thm}
\label{theappmasslem}
If $\cH_{1:n}\subset \calL_{S_{1:n}}^{\otimes}$ is a family of sequential functions, then for every $z_{1:n}\in S_{1:n}^{\otimes}$ and every $r>0$,
\begin{align}
\Radcemp(\cH_{1:n},z_{1:n})\leq rn+|H_{1:n}(z_{1:n})|_{n,2}\sqrt{2 \log (\cN_{1}(\cH_{1:n},z_{1:n},r) )}.
\label{equbouempradcomcovnum}
\end{align}
where $H_{1:n}:S_{1:n}^{\otimes}\to [0,\infty]^{n}$ is the sequential function defined by \eqref{equdefseqfuncaph}.
\end{thm}
\proof 
The conclusion is obvious when some $H_{k}(z_{k})=\infty$ or when $\cN_{1}(\cH_{1:n},z_{1:n},r)=\infty$. In any case, it is easy to see that any $a_{1:n}\in \calL_{S_{1:n}}^{\otimes}$ satisfying
\begin{align}
\label{equl1covine}
|a_{1:n}-h_{1:n}|_{z_{1:n},1}\leq rn
\end{align} 
for some $h_{1:n}\in \cH_{1:n}$, can be assumed  to satisfy that 
\begin{align}
\label{equasstru}
|a_{k}(z_{k})|\leq H_{k}(z_{k}),&\qquad \mbox{for $k\in \{1,\dots,n\}$}
\end{align}
 (truncate by components if necessary). This implies in particular that there exists a (minimal)  
 \begin{align}
 \{a_{1:n}^{(k)}\}_{k=1}^{\cN_{1}(\cH_{1:n},z_{1:n},r)}\subset \calL_{S_{1:n}}^{\otimes}
 \end{align}
   satisfying \eqref{equasstru} for all $a_{1:n}^{(k)}$ and \eqref{equl1covine} for all $h_{1:n}\in \cH_{1:n}$ and some $a_{1:n}^{(k_{h_{1:n}})}$. The conclusion follows from  Lemma \ref{lemmaslem} and the inequality
   \begin{align}
   u_{1:n}\cdot h_{1:n}(z_{1:n})=&u_{1:n}\cdot(h_{1:n}(z_{1:n})-a_{1:n}^{(k_{h_{1:n}})}(z_{1:n}))+u_{1:n}\cdot a_{1:n}^{k_{h_{1:n}}}(z_{1:n})\notag\\
   \leq& r n+\max_{k}(u_{1:n}\cdot a_{1:n}^{(k)}(z_{1:n})).~\finproof
   \end{align}

Combined with \eqref{equradestconint}, this has the following consequence, fundamental to the application of these ideas to error bounds in learning
\bcor
Under the hypotheses of Corollary \ref{cortheradbou}, the inequality
\begin{align}
&\sum_{k=1}^{n}\Esp{(\hat{h}_{k}(Z_{k}')-\tilde{h}_{k}(Z_{k}'))|Z_{1:n}}\notag\\
\leq &  2(r+||\,|H_{1:n}(Z_{1:n})|_{n,2}||_{\PP,\infty} \big(
\sqrt{2\log({2}/{\delta})}+ \,\Esp{\sqrt{2\log(N_{1}(\cH_{1:n},Z_{1:n},r/n))}}\big)\big)
\label{equradestconintmas}
\end{align}
holds for every $(r,\delta)\in (0,\infty)\times(0,1)$, with probability at least $1-\delta$.\ecor
\begin{proof} 
The estimate \eqref{equbouempradcomcovnum} with $r/n$ in place of $r$ gives, by integration and H\"older's inequality, the bound
\begin{align}
&\Radcave(\cH_{1:n},Z_{1:n})=\Esp{\Radcemp(\cH_{1:n},Z_{1:n})}\notag\\
\leq
&  r+\Esp{|H_{1:n}(Z_{1:n})|_{n,2}\sqrt{2 \log (N_{1}(\cH_{1:n},Z_{1:n},r/n) )}}\notag\\
\leq &  r+||\,|H_{1:n}(Z_{1:n})|_{n,2}||_{\PP,\infty}\Esp{\sqrt{2 \log (N_{1}(\cH_{1:n},Z_{1:n},r/n) )}}.
\label{equproradcomentestalm}
\end{align}
A combination of \eqref{equproradcomentestalm} with \eqref{equradestconint} gives \eqref{equradestconintmas}.\fe
\end{proof}

\brem
\label{remprebouconhul}
It follows from the inequalities in Proposition \ref{proproradcom} and an elementary argument that if, in the above, $\cH_{1:n}$ is a family satisfying
\begin{align}
\cG_{1:n} \subset \cH_{1:n}\subset co(\cG_{1:n})&\qquad\mbox{[resp. $\cG_{1:n} \subset \cH_{1:n}\subset cobal(\cG_{1:n})$]}
\end{align}
(see \eqref{equdefconhul} and \eqref{equdefconpri} for definitions), and if $G_{1:n}$ is defined as in \eqref{equdefseqfuncaph} with $\cH_{1:n}$ replaced by $\cG_{1:n}$ (actually it is easy to see that $G_{1:n}=H_{1:n}$), then all the instances of $\Radcave(\cH_{1:n},Z_{1:n})$ and $H_{1:n}$ in the inequalities above can be replaced by instances of $\Radcave(\cG_{1:n},Z_{1:n})$ and $G_{1:n}$ [resp. by instances of $2\Radcave(\cG_{1:n},Z_{1:n})$ and $G_{1;n}$]. An examination of the proofs above shows that this implies in particular that:
\erems

\bcor
\label{corgenpowconhul}
The generalization bounds \eqref{equboutheupebougen}, \eqref{equdeverrlosradcom}, \eqref{equradestconint} and \eqref{equradestconintmas} remain true if $\hat{h}$ and $\tilde{h}$ in \eqref{empandglosolgen} and \eqref{empandglosolgenone} are computed with $\cH_{1:n}$ replaced by any sequential family $\cF_{1:n}$ with
\begin{align}
\cH_{1:n}\subset \cF_{1:n}\subset co(\cH_{1:n})
\end{align}
(the change when $\cH_{1:n}\subset \cF_{1:n}\subset cobal(\cH_{1:n})$ is evident also from Remark \ref{remprebouconhul}).
\ecor
For an application related (but not identical) to this fact within the case of neural networks see Example \ref{exaneunetgenpow} below.

\subsection{Entropy estimates}
\label{secentest}
Of special importance for the applications is the notion of entropy estimates, which permit to give explicit upper bounds to the generalization power of regression schemes via \eqref{equradestconintmas}:
\begin{dfn}
\label{defentest}
Let $\cH_{1:n}\subset \calL_{S_{1:n}}^{\otimes}$ be a family of sequential functions and let  $Z_{1:n}$ be a random element  of $S_{1:n}^{\otimes}$. An {entropy estimate}  of $\cH_{1:n}$ with respect to $Z_{1:n}$ is a Borel-measurable, nonincreasing function $L_{\cH_{1:n},Z_{1:n}}:[0,\infty)\to [0,\infty]$ such that for all $r>0$,
\begin{align}
\log(\Esp{ N_{1}(\cH_{1:n},Z_{1:n},r)})\leq L_{\cH_{1:n},Z_{1:n}}(r).
\end{align}
If the stronger condition 
\begin{align}
\label{equuppbouunientest}
||\log (N_{1}(\cH_{1:n},Z_{1:n},r))||_{\PP,\infty}\leq L_{\cH_{1:n},Z_{1:n}}(r)
\end{align}
holds, we call $L_{\cH_{1:n},Z_{1:n}}$ a uniform entropy estimate (of $\cH_{1:n}$ with respect to $Z_{1:n}$).

If $\cH\subset \calL(S)$ is given and $Z_{1:n}$ is a random element of $S^{n}$ we define entropy estimates of $\cH$ at $Z_{1:n}$ via the family $diag(\cH)_{1:n}$ in \eqref{equdefdiafam}.\fe
\end{dfn}

The following two are instances of uniform entropy estimates that do not depend on the distribution of $Z_{1:n}$:
\begin{exa}
\label{exasaushe}
For a function $f:\R^{d}\to \R$, define the {\it subgraph of $f$} as the set
\begin{align}
G_{f}^{+}\eqdef \{(x,y)\in \R^{d}\times \R: y\leq f(x)\}.
\end{align}
The {\it VC-dimension $V_{\cF}$ of a family $\cF$ of functions $\R^{d}\to \R$}  {is the supremum of the natural numbers $l$ with the following property: there exists a set $G\subset \R^{d}\times \R$ with $l$ elements such that every subset $G'\subset G$ can be written in the form $G'=G\cap G_{f}^{+}$ for some $f\in \cF$}.  

When $\cF$ is a family of bounded, nonnegative functions $f:\R^{d}\to [0,B]$, one has the following uniform $L^{1}-$entropy estimate (\cite[Lemma 9.2 and Theorem 9.4.]{gyor:kohl:krzy:walk:02}) for $\cF$: for every $r\in (0,B/4]$ and every $z_{1:n}\in (\R^{d})^{n}$,
\begin{align}
\log (N_{1}(\cF,z_{1:n},r))& \leq L(r)\notag\\
& \eqdef \log 3+V_{\cF}(1+\log 2 +\log ({B}/{r}) +\log(1+\log 3 +\log ({B}/{r}))),
\label{equdefunientestfinvc}
\end{align} 
which is clearly $O(\log(1/r))$ as $r\to 0^{+}$ when $V_{\cF}<\infty$\footnote{Note that the restriction  $r\in [0,B/4]$ can be easily bypassed in order to extend $L_{V_{\cF},B}$ to a function $[0,\infty)\to [0,\infty)$: one can for instance  define $L_{V_{\cF}}(r)=0$ if $r>B$, and one can substigute $L_{V_{\cF},B}(r)\mapsto L_{V_{\cF},4B}(r)$, to obtain an estimate valid for $r\in [0,4B/4]=[0,B]$). A similar trick allows us to give uniform entropy estimates via \eqref{equdefunientestfinvc} on (perhaps   nonpositive)  families $\cF$ of functions $f:\R^{d}\to [-B,B]$: the family $\cF'=\cF+B\eqdef \{f+B:f\in \cF\}$ has the same covering numbers as $\cF$, satisfies $V_{\cF}=V_{\cF'}$, and its elements are functions $f:\R^{d}\to [0,2B]$.}, in particular when $\cF=T_{B}{\cG}$ is the family of truncated functions from a  vector space of dimension $d_{\cG}<\infty$, thanks to the bounds 
\begin{align}
V_{T_{B}\cG}\leq V_{\cG}\leq d_{\cG}+1
\end{align}
 (\cite[Theorem  9.5 (and previous paragraph) and Equation (10.23)]{gyor:kohl:krzy:walk:02}).\fe
\end{exa}
The  estimate \eqref{equdefunientestfinvc} is a consequence of the celebrated {\it Sauer-Shelah lemma} (see \cite[Theorem 8.3.16 p.193]{vershynin2018high}). It  is therefore a relationship between the  {\it complexity of $\cF$}, as measured by $V_{\cF}$, and the notion of uniform entropy estimates.

\begin{exa}
\label{exaneunet}
A second example is given by neural networks with one layer and $N$ ``independently powered'' units: it is  shown in  \cite[p.314]{gyor:kohl:krzy:walk:02} that if $\sigma:\R\to [0,1]$ is any cumulative distribution function (for instance a ``sigmoid'' function with asymptotes $y=0$ and $y=1$) and 
\begin{align}
\cG=\cG(\sigma,N,B)
\end{align}
 is the family of functions $g:\R^{d}\to \R$ of the form
\begin{align}
\label{equdefneunet}
g(x)= c_{0}+ \sum_{k=1}^{N}c_{k}\sigma(a_{k}^{T}x+b_{k}) 
\end{align}
with $N\in \N$ fixed,  and with 
\begin{align}
\label{equconindpowuni}
(a_{1:N}, b_{1:N}, c_{0:N})\in  (\R^{d})^{N}\times \R^{N}\times [-B,B]^{N+1}\,\,\mbox{ for some $B>0$,}
\end{align}
%\begin{align}
%\max_{k}|c_{k}|\leq B&\qquad \mbox{for some $B>0$},
%\end{align}
 then\footnote{The claim proved in \cite[p.314]{gyor:kohl:krzy:walk:02} is about the family $\cF_{(B,N,\sigma)}\subset \cG_{(B,N,\sigma)}$ of neural networks with ``jointly powered'' units, determined by the restriction $|c_{0:N}|_{N+1,1}\leq B$, but it is clear that the argument works for $ \cG_{(B,N,\sigma)}$.} for every $(z_{1:n},r)\in (\R^{d})^{n}\times (0,B/2)$,
\begin{align}
\log (N_{1}(\cG,z_{1:n},r))& \leq L_{\mathcal{NN},N,B}(r)\notag\\ 
&:=((2d+5)N+1)(1+\log(12)+\log(B/r)+\log(N+1)).\fe\label{equdefunientestfinvcneunet}
\end{align}  
\end{exa}

Notice that both the estimates \eqref{equdefunientestfinvc} and \eqref{equdefunientestfinvcneunet} are of the form $O(\log(1/r))$ as $r\to 0^{+}$. Inspired by the convention (within the theory of  empirical processes) proposed in (\cite[P.38]{pol}), this behavior can be formalized via the following definition:

\begin{dfn}
\label{defeucsubeuc}
Let $\cH_{1:n}\subset \calL_{S_{1:n}}^{\otimes}$ be a family of sequential functions and let  $Z_{1:n}$ be a random element  of $S_{1:n}^{\otimes}$
\begin{enumerate}
\item $(\cH_{1:n},Z_{1:n})$ is called [uniformly] subeuclidean if it admits a [uniform] entropy estimate $L_{\cH_{1:n},Z_{1:n}}$ with
\begin{align}
L_{\cH_{1:n},Z_{1:n}}(r)=O(\log(1/r)), &\qquad \mbox{as $r\to 0$.}
\end{align}
\item $(\cH_{1:n},Z_{1:n})$ is called [uniformly] euclidean of order $\alpha>0$ if it admits a [uniform] entropy estimate $L_{\cH_{1:n},Z_{1:n}}$ with
\begin{align}
L_{\cH_{1:n},Z_{1:n}}(r)=O((1/r)^{\alpha}), &\qquad \mbox{as $r\to 0$.}\fe
\end{align}
\end{enumerate} 
\end{dfn}

To further illustrate the interaction  between entropy estimates and Rademacher complexities, let us finally give the following:

\begin{exa}
\label{exaneunetgenpow}
Consider the family 
\begin{align}
\cF=\cF(\sigma,N,B)
\end{align}
 of ``jointly powered'' (or ``lasso-regularized'') neural networks, which we define as in \eqref{equdefneunet}, \eqref{equconindpowuni} but with the stronger restriction
\begin{align}
|c_{0:N}|_{N+1,1}\leq B
\end{align}
(see \eqref{equdeflpnor}). It is easy to see that, for $\cG(\sigma,N,B)$ as in Example \ref{exaneunet},
\begin{align}
\label{equequwithconhul}
\cF(\sigma,1,B)=\cG(\sigma,1,B)\subset \cF(\sigma,N,B)\subset co(\cG(\sigma,1,B))
\end{align}
 for every (number of units) $N$. If, furthermore, 
 \begin{align}
 \cH=\cH(\cF(\sigma,N,B))
 \end{align}
 is the family of losses associated to  quadratic regression over $\cF(\sigma, N,B)$ with truncated response, namely
 \begin{align}
 \cH:=\{\R^{d}\times \R \ni(x,y)\mapsto h_{f}(x,y):=(T_{B}y-f(x))^{2}: f\in \cF(\sigma,N,B)\},
 \end{align}
 then since 
 \begin{align}
 (u,y)\mapsto (T_{B}y-u)^{2}
 \end{align}
 is $2B-$Lipschitz in $[-B,B]\times \R$, Talagrand's contraction lemma (\cite[Lemma 4.2]{mohri2018foundations}) implies that, for every $(x,y)_{1:n}\in ((\R^{d})\times \R)^{n}$
 \begin{align}
 \Radcemp(\cH,(x,y)_{1:n})\leq& 2B\Radcemp(\{(x,y)\mapsto(y-f(x)):f\in \cF(\sigma,N,B)\}, (x,y)_{1:n})\notag\\
 \leq &2B\Radcemp(\cF(\sigma,N,B),x_{1:n})\notag\\
 = & 2B\Radcemp(\cG(\sigma,1,B),x_{1:n})
 \label{equalmprogenneunet}
 \end{align}
 where in the last equality we used \eqref{equequwithconhul} and Proposition \ref{proproradcom}. Combining in the respective order 
  \eqref{equradestconint}, \eqref{equuppbourad}, \eqref{equalmprogenneunet}, Theorem \ref{theappmasslem} with $diag(\cG(\sigma,1,B))_{1:n}$ in place of $\cH_{1:n}$ and $r=B/\sqrt{n}$, \eqref{equdefunientestfinvcneunet} (w.l.o.g $n>4$), and the  inequality 
  \begin{align}
  \sup_{(g,x,y)\in \cG(\sigma,1,B)\times \R^{d}\times \R}|T_{B}y-g(x)|^{2}\leq 4B^{2}
  \end{align}
 we arrive at the following:
 
 \begin{thm}
 \label{thegenneunet}
 If $\cF(\sigma,N,B)$ is as above,  $(X,Y)_{0:n}$ is an i.i.d. sequence of random variables in $\R^{d}\times \R$, and 
 \begin{align}
 \hat{f}\in \arg\min_{f\in \cF(\sigma,N,B)} \sum_{k=1}^{n}(T_{B}Y_{k}-f(X_{k}))^{2},
 \end{align} 
 then for every $\delta\in (0,1)$
 \begin{align}
 \Esp{(T_{B}Y_{0}-\hat{f}(X_{0}))^{2}|(X,Y)_{1:n}}
 \leq& \min_{f\in \cF(\sigma,B,N)} \Esp{(T_{B}Y_{0}-f(X_{0}))^{2}}\notag\\
 +&\frac{B^{2}}{\sqrt{n}}\left(1+\sqrt{2}\left(8\sqrt{\log(2/\delta)}+\sqrt{(2d+6)(\log(24e\sqrt{n}))}\right)\right)
 \label{equgenpowneunet}
 \end{align}
 with probability at least $1-\delta$.
 \end{thm}
 
 We will see in the next section that this bound is suboptimal for large samples; in addition, the constants in \eqref{equgenpowneunet} can be improved  applying the consequences of Remark \ref{remhoa} (essentially replace the $8$ by  a $4$). In any case, since the right-hand side of \eqref{equgenpowneunet}  does not depend on the number of units $N$, Theorem \ref{thegenneunet} is consistent\footnote{For the case of one-layer neural networks, but the multi-layer case is not difficult to analyze departing from the arguments in \cite{gyor:kohl:krzy:walk:02}.} with the popular  observation that {lasso regularization preserves the generalization power of neural networks as the number of units increases\footnote{See \cite{kakshatew} for a discussion of the role of regularization in the context of linear regression. For a discussion specialized in neural networks, see \cite[Section 2.6]{wol}.}}.
 
 \end{exa}

\section{VC-Confidence intervals for truncated least squares  
}
\label{secgenvc}

In this section we present some error estimates in probability for least squares regression schemes that rely on versions of \eqref{equdevsupind} adjusted to the case in which $Z_{1:n}=(X,Y)_{1:n}$ is an independent sequence in $\stax\times \R$ and the sequential space $\cH_{1:n}$ is the space of pointwise  deviations  of the losses associated to a least square regression scheme on a space of hypotheses $\cG\eqdef \{g:\stax\to \R\}$.

\subsection{Setting and additional notation}
\label{secsetaddnot}
The setting is the following: for some $B>0$, consider the truncation operator
\begin{align}
\label{equdeftruope}
T_{B}y\eqdef \min\{\max\{y,-B\},B\},
\end{align} 
and for $(X,Y)_{1:n}$ as before, let $\phi_{B,k}:\stax\to [-B,B]$ be such that 
 \begin{align}
 \label{equchaphik}
 \phi_{B,k}(X_{k})=\Esp{T_{B}Y_{k}|X_{k}},&\qquad \PP-a.s. 
 \end{align}
 
  Let $\cG\subset \calL(S)$, assume that there exist 
\begin{align}
\hat{g}\in \arg\min_{g\in \cG} \sum_{k=1}^{n}|g(X_{k})-T_{B}Y_{k}|^{2},&\qquad\tilde{g}\in \arg\min_{g\in \cG} \sum_{k=1}^{n}\Esp{|g(X_{k})-T_{B}Y_{k}|^{2}}
\label{equdefgapp}
\end{align}
and define
\begin{align}
\hat{\phi}_{B}\eqdef T_{B}\hat{g}
\end{align}

We define, for every $(g,k)\in \calL(S)\times \{1,\dots,n\}$, 
\begin{align}
\label{thefuntalinapp}
h_{g,k}(x,y)\eqdef |g(x)-T_{B}y|^{2}-|\phi_{B,k}(x)-T_{B}y|^{2},
\end{align}
which in particular satisfies, by \eqref{equchaphik} and the Pythagorean theorem, the equation
\begin{align}
\label{equconpyt}
\Esp{h_{g,k}(X_{k},T_{B}Y_{k})}=\Esp{|g(X_{k})-\phi_{B,k}(X_{k})|^{2}}
\end{align}

\subsection{The conditional deviation of the $L^{2}-$error}
In this section, we will describe some estimates of the conditional probability 
\begin{align}
\PPro{\Esp{\sum_{k=1}^{n}  (h_{\hat{\phi}_{B},k}(Z_{k}')-h_{\tilde{g},k}(Z_{k}')) |Z_{1:n}}
>n\epsilon}
\end{align}
where $Z_{1:n}'=(X',Y')_{1:n}$ is an independent copy of $Z_{1:n}$. Notice that, if we (naturally) denote by $y:\stax\times \R\to \R$ the projection $y(x_{0},y_{0})=y_{0}$, then 
this is the same as 
\begin{align}
\PPro{
\frac{1}{n}\sum_{k=1}^{n} ||\hat{\phi}_{B}-T_{B}y||_{k}^{2}-\inf_{g\in \cG}\frac{1}{n}\sum_{k=1}^{n}||g-T_{B}y||_{k}^{2}
>\epsilon}\notag\\
= \PPro{
\frac{1}{n}\sum_{k=1}^{n} ||\hat{\phi}_{B}-\phi_{B,k}||_{k}^{2}-\inf_{g\in \cG}\frac{1}{n}\sum_{k=1}^{n}||g-\phi_{B,k}||_{k}^{2}
>\epsilon}
\label{equaltdesdev}
\end{align}
where $||\,\cdot\, ||_{k}$ at the left [resp. right]--hand  side of \eqref{equaltdesdev} is the $L^{2}$ norm on $\calL(\stax\times \R)$ [resp. $\calL(\stax)$] associated to the law of $(X_{k},Y_{k})$ [resp. $X_{k}$], and where we used \eqref{equconpyt} for claiming the equality.

 The estimates will depend on the functions
\begin{align}
A:&S^{n}\times(1,\infty)\times (1,\infty)\times (0,\infty) \to (0,\infty]\notag\\
a:&(1,\infty)\times (1,\infty)\times (0,\infty) \to (0,\infty]\notag\\
\epsilon_{n}:& (1,\infty)\times (1,\infty) \to (0,\infty)\notag\\
b:&(1,\infty)\times (1,\infty) \to (0,\infty)
 \end{align}
given by
\begin{align}
A(x_{1:n},c,\lambda,\epsilon)\eqdef &2(c+1)(2c+3){\cN_{1}({T_{B}\cG}, x_{1:n},\frac{1}{2^{5}}\frac{1}{B}\frac{1}{\lambda(c-1)+1}(1-\frac{1}{c})\epsilon)},\label{equspearan}\\
&\notag\\
a(c,\lambda,\epsilon)\eqdef &
\Esp{A(X_{1:n},c,\lambda,\epsilon)}
,\label{equspea}\\
&\notag\\
\epsilon_{n}(c,\lambda)\eqdef &8B^{2}({-(\lambda-1)+\sqrt{(\lambda-1)^{2}+c(c+1)\lambda^{2}/n}})\, \label{equspee} \\
&\notag\\
b(c,\lambda)\eqdef &
\frac{1}{2^{5}B^{2}}\frac{1}{(\frac{1}{3}(1-\frac{1}{c})(1-\frac{1}{\lambda})+(2\lambda-1))^{2}}(1-\frac{1}{c})^{{3}}(1-\frac{1}{\lambda}),
\label{equspeb}
\end{align}
where $T_{B}\cG\eqdef \{T_{B}g:g\in \cG\}$ is the family of $B-$truncated functions from $\cG$ (see \eqref{equdeftruope}).

\begin{thm}
\label{theratconlsrpro}
Under the previous setting, for every $(c,\lambda,\delta)\in(1,\infty)\times(1,\infty) \times (0,1)$, the inequality
\begin{align}
\Esp{
\sum_{k=1}^{n}  (h_{\hat{\phi},k}(Z_{k}')-h_{\tilde{g},k}(Z_{k}'))\left |Z_{1:n}\right.}\leq &
\sum_{k=1}^{n} (\lambda h_{\tilde{g},k}(Z_{k})-\Esp{h_{\tilde{g},k}(Z_{k})}) \notag\\
+&n \epsilon_{n}(c,\lambda)\lor \frac{1}{b(c,\lambda)}(\log a(c,\lambda,\epsilon_{n}(c,\lambda))+\log(1/\delta)))
\label{equbuoproreg}
\end{align}
holds with probability at least $1-\delta$.
\end{thm}

\begin{proof}
Let $\lambda>1$ be given and write
\begin{align}
&\sum_{k=1}^{n}\Esp{h_{\hat{\phi},k}(Z_{k}')-h_{\tilde{g},k}(Z_{k}')\left |Z_{1:n}\right.}\notag\\
=&
\sum_{k=1}^{n}(\Esp{h_{\hat{\phi},k}(Z_{k}')|Z_{1:n}}-\lambda h_{\hat{\phi},k}(Z_{k}))+ 
\sum_{k=1}^{n} (\lambda h_{\hat{\phi},k}(Z_{k})-\Esp{h_{\tilde{g},k}(Z_{k})})\notag\\
\leq&
\sum_{k=1}^{n}(\Esp{h_{\hat{\phi},k}(Z_{k}')|Z_{1:n}}-\lambda h_{\hat{\phi},k}(Z_{k}))+ \sum_{k=1}^{n} (\lambda h_{\tilde{g},k}(Z_{k})-\Esp{h_{\tilde{g},k}(Z_{k})})\notag\\
\leq &
\sup_{g\in \cG}\left\{
\sum_{k=1}^{n}(\Esp{h_{{T_{B}g},k}(Z_{k})}-\lambda h_{T_{B}g,k}(Z_{k}))\right\}+ 
\sum_{k=1}^{n} (\lambda h_{\tilde{g},k}(Z_{k})-\Esp{h_{\tilde{g},k}(Z_{k})}).
\label{inidesprolardev}
\end{align}
As shown in \cite[Section 3.2, P. 67]{bargob19} (after an elementary homogenization, see for instance  \cite[Section 3.2, P. 69]{bargob19}), the functions \eqref{equspea}, \eqref{equspee} and \eqref{equspeb} are such that, for every $(c,\lambda, \epsilon)\in(1,\infty)\times (1,\infty)\times (0,\infty)$,
\begin{align}
\PPro{\sup_{g\in \cG}\left\{\sum_{k=1}^{n}(\Esp{h_{{T_{B}g},k}(Z_{k})}-\lambda h_{T_{B}g,k}(Z_{k}))\right\}\geq n\epsilon} & \notag\\ 
\leq \un_{\{\epsilon< \epsilon_{n}(c,\lambda)\}}+&a(c,\lambda,\epsilon)\exp(-b(c,\lambda) n \epsilon )\un_{\{\epsilon\geq \epsilon_{n}(c,\lambda)\}}\notag\\
\leq\un_{\{\epsilon< \epsilon_{n}(c,\lambda)\}} +& 
a_{n}(c,\lambda)\exp(-b(c,\lambda) n \epsilon )\un_{\{\epsilon\geq \epsilon_{n}(c,\lambda)\}}\notag\\
\label{lardevinesym}
\end{align}
where
\begin{align}
\label{equdefan}
a_{n}(c,\lambda)\eqdef a(c,\lambda,\epsilon_{n}(c,\lambda)).
\end{align}
From here on, we keep fixed $(c,\lambda)$ and therefore we omit for simplicity the arguments of the functions of these parameters (so for instance $a_{n}\equiv a_{n}(c,\lambda)$).

From the equivalence
\begin{align}
a_{n}\exp(-bn\epsilon)\leq \delta \iff \frac{1}{bn}(\log a_{n} +\log (1/\delta))\leq \epsilon,
\end{align}
and from \eqref{lardevinesym}, it follows that for any $\delta\in (0,1)$ and for
\begin{align}
\label{equdefepsdel}
\epsilon_{n}(\delta)\eqdef \frac{1}{bn}(\log a_{n} +\log (1/\delta))\lor \epsilon_{n},
\end{align}

\begin{align}
\sup_{g\in \cG}\left\{\sum_{k=1}^{n}(\Esp{h_{{T_{B}g},k}(Z_{k})}-\lambda h_{T_{B}g,k}(Z_{k}))\right\}
\leq n \epsilon_{n}(\delta) 
 \end{align}
with probability at least $1-\delta$: this and \eqref{inidesprolardev} imply the claim of the theorem.\fe
\end{proof}

\brem
\label{remdepbouvc}
In the setting of Section \ref{secremdep},  using (for simplicity) the notation
\begin{align}
m_{n,\delta}=%\left\lceil
\log_{r}(n/\delta)%\right\rceil
,& \qquad n_{m,\delta}:=%\left\lfloor 
n/m(n,\delta)
%\right\rfloor,
\end{align}
(we assume also for simplicity that $m(n,\delta)\in \mathbb{Z}$ and divides $n$), and extending the definition \eqref{equspearan}  of $A$ to arbitrary finite sequences $(z_{j})_{j\in J}$ by means of the empirical measure $
\sum \delta_{z_{j}}/|J|$, we obtain this time, under \eqref{equexpmixrat}, the upper bound
\begin{align}
\un_{\{\epsilon< m_{n,\delta}\epsilon_{n_{m,\delta} }(c,\lambda)\}} +& 
(m_{n,\delta}a_{n_{m,\delta} }^{*}(c,\lambda)\exp(-b(c,\lambda) n_{m,\delta} \epsilon )+\delta)\un_{\{\epsilon\geq m_{n,\delta} \epsilon_{n_{m,\delta}}(c,\lambda)\}}
\end{align}
for the left-hand side of \eqref{lardevinesym}, where this time
\begin{align}
a_{n_{m,\delta}}^{*}=\max_{1\leq k\leq m_{n,\delta} }\Esp{A(Z_{J_{k}}^{*},c,\lambda,\epsilon_{n_{m,\delta} })}
\end{align}
with $J_{k}:=\{k+lm_{n,\delta}\}_{l}\cap\{1,\dots,n\}$ ($1\leq k\leq m_{n,\delta}$) and $Z_{1:n}^{*}$ an independent sequence with the same marginals as $Z_{1:n}$. The conclusion (ignoring divisibility issues) is that of Theorem \ref{theratconlsrpro} with the second line of \eqref{equbuoproreg} replaced by
\begin{align}
\label{equdefepsdelmod}
 (n_{m,\delta/2} m_{n,\delta/2}\epsilon_{n_{m,\delta/2}})\lor\frac{1}{b(c,\lambda)}(\log m_{n,\delta/2}+\log a_{n_{m,\delta/2}}^* +\log (2/\delta)
).
\end{align}
\erems

\subsection{Remarks on  the bound in Theorem \ref{theratconlsrpro}
}
\label{secrembou}
Let us give  explicit (and useful) estimates of the second term at the right hand side  of \eqref{equbuoproreg}. 
We begin by pointing out the estimate
\begin{align}
\epsilon_{n}(c,\lambda)\leq 8B^{2}\lambda\left(\left(\frac{c(c+1)}{2n}\frac{\lambda}{\lambda-1}\right)\wedge \sqrt{\frac{c(c+1)}{n}}\right)
\label{equupeesteps}
\end{align}
which follows from elementary majorizations after multiplying and dividing the sum in \eqref{equspee} by its conjugate.

Notice now that 
\begin{align}
& \frac{1}{2^{5}}\frac{1}{B}\frac{1}{\lambda(c-1)+1}(1-\frac{1}{c})\epsilon_{n}(c,\lambda) & \notag\\
=&\frac{B}{2^{2}}\frac{(c^{2}-1)\lambda^{2}}{(\lambda(c-1)+1)((((\lambda-1)n)^{2}+c(c+1)\lambda^{2}n)^{1/2}+(\lambda-1)n)} & \notag\\
\geq & \frac{B}{2^{2}(2^{1/2}+1)}\left(1-\frac{1}{c}\right)\frac{1}{n} &
\label{equexpazersim}
\end{align}
(this minorization may be suboptimal: see the discussion leading  to \eqref{equsecsimuppbouproref} below). Simple manipulations show also that, if
\begin{align}
p(c)\eqdef \frac{c}{c-1}, & &\\
q_{1}(\lambda)\eqdef \frac{1}{9}\left(1-\frac{1}{\lambda}\right), &\qquad q_{2}(\lambda)\eqdef \frac{2}{3}(2\lambda-1), & q_{3}(\lambda)\eqdef (2\lambda-1)^{2}\frac{\lambda}{\lambda-1}
\end{align}
then
\begin{align}
\frac{1}{b(c,\lambda)}=&
2^{5} B^{2}(q_{1}(\lambda)p(c)+q_{2}(\lambda)p^{2}(c)+q_{3}(\lambda) p^3(c)).
\label{equexponeovebsim}
\end{align}

Together, \eqref{equupeesteps}, \eqref{equexpazersim} and \eqref{equexponeovebsim} permit to see that, if 
\begin{align}
q_{0}(\lambda)\eqdef \frac{1}{8}\frac{\lambda^{2}}{(\lambda-1)}, 
\end{align}
then the second term at the right-hand side of \eqref{equbuoproreg} is upper bounded by
\begin{align}
&2^{5}\frac{B^{2}}{n}\left(\frac{}{}(\frac{}{}q_{0}(\lambda) c(c+1))
\right.\lor\left(\left(\sum_{k=1}^{3}q_{k}(\lambda)(p(c))^{k}\right)\left(\log\left(\frac{2(c+1)(2c+3)}{\delta}\right)\right.\right.\notag\\
+&\left.\left. \left.\log\left(\Esp{N_{1}\left({T_{B}\cG}, Z_{1:n},\frac{B}{2^{2}(2^{1/2}+1)}\left(1-\frac{1}{c}\right)\frac{1}{n}\right)}\right) \right)\right)\right).\notag\\
\label{equfirsimuppboupro}
\end{align}
A numerical search shows that the function $(1,\infty)\times (1,\infty)\to \R$ given by
\begin{align}
(c,\lambda)\mapsto V(c,\lambda)=2^{5}\left(q_{0}(\lambda) c(c+1)
\lor\left(\left(\sum_{k=1}^{3}q_{k}(\lambda)(p(c))^{k}\right)\left(\log\left({2(c+1)(2c+3)}\right)\right)\right)\right)
\end{align}
attains a local minimum at some $(\lambda_{0},c_{0})$ with 
\begin{align}
\label{equdefminvcbou} 
1.29<\lambda_{0}<1.3, &\qquad 11.46<c_{0}<11.47
\end{align}
and with 
\begin{align}
\label{equdefcoeminvcbou}
3291<V(c_{0},\lambda_{0})< 3292,
\end{align}
thus giving the bound
\begin{align}
\frac{1}{n}\sum_{k=1}^{n} (\lambda_{0} h_{\tilde{g},k}(Z_{k})-\Esp{h_{\tilde{g},k}(Z_{k})})\notag\\
+3292\frac{B^{2}}{n}\left(1+\log\left(\frac{1}{\delta}\right)+\log\left(\Esp{N_{1}\left({T_{B}\cG}, X_{1:n},0.094\frac{B}{n}\right)}\right)\right)
\label{equfirsimuppbouproopt} 
\end{align}
for the left-hand side of \eqref{equbuoproreg}.

{\it Behavior as $\lambda\to 1$.} In applications, we may be interested in upper bounding the left hand side of \eqref{equbuoproreg} for $\lambda$ close to one\footnote{See for instance \eqref{equbuoproregsim2} below, where it is necessary to take $\lambda$ close to one in order to guarantee an upper bound close to the approximation error (for $n$ large enough).}. To discuss the behavior of our bounds for this case, notice that
\begin{align}
\sum_{k=1}^{3} q_{k}(\lambda) p^{k}(c)= p^{3}(c)\leq \left(\sum_{k=1}^{3} q_{k}(\lambda)\right)p^{3}(c)\defeq \frac{C(\lambda)}{\lambda-1}p^{3}(c) 
\end{align} 
where 
\begin{align}
C(\lambda)\eqdef (\lambda-1)\sum_{k=1}^{3} q_{k}(\lambda)\to_{\lambda\to 1^{+}} 1
\end{align}
This, together with the inequality $q_{0}(\lambda)\leq q_{3}(\lambda)$, allows us to upper bound \eqref{equfirsimuppboupro} by 
\beql{equsecsimuppboupro}
&2^{2}\frac{1}{\lambda-1}\frac{B^{2}}{n}\left(\lambda^{2}c(c+1)\lor\left(2^{3}C(\lambda)\left(\frac{c}{c-1}\right)^{3}\left(\log\left(\frac{2(c+1)(2c+3)}{\delta}\right)\right.\right.\right.\\&\qqq\left.\left.\left.
+\log\left(\Esp{N_{1}\left({T_{B}\cG}, Z_{1:n},\frac{B}{12}\left(1-\frac{1}{c}\right)\frac{1}{n}\right)}\right)\right)\right)\right). 
\eeql
with $C(\lambda)\to_{\lambda\to 1^{+}} 1$. The expression \eqref{equsecsimuppboupro} can be optimized in several directions according to necessity. To illustrate concretely, the restriction
\begin{align}
\label{equfirreslam}
1<\lambda\leq 13/12
\end{align}
implies the bound
\begin{align}
\label{equfirbouclam}
C(\lambda)<2.
\end{align}
Using \eqref{equfirreslam} together with \eqref{equfirbouclam} and with the choice $c=2$, we can bound the second term at the right-hand side of \eqref{equbuoproreg} by
\begin{align}
2^{6}\frac{1}{\lambda-1}\frac{B^{2}}{n}\left(\log({42})+\log\left(\frac{1}{\delta}\right)
+\log\left(\Esp{N_{1}\left({T_{B}\cG}, Z_{1:n},\frac{B}{24\,n}\right)}\right)\right).
\label{equsecsimuppbouprosecbou}
\end{align}
provided \eqref{equfirreslam}.

\medskip

{\it A refined minorization}. Let us finally discuss a case in which a more refined use of the estimate \eqref{equbuoproreg} gives an interesting consequence. The setting is the same as before but this time we will make more explicit the dependence on $n$, thus denoting $B\equiv B_{n}$, $Z_{1:n}\equiv Z_{1:n}^{(n)}$, etc.

If we have in this setting a sequence $(F_{n}(\cdot))_{n}$ of {entropy estimates for $(T_{B_{n}}\cG_{n},Z_{1:n}^{(n)})_{n}$} (Definition \ref{defentest}) and if the following ``continuity'' property holds for $(F_{n})_{n}$
\begin{align}
\frac{r_{n}}{r_{n'}}\to_{n} 1
\mbox{\,\,\,\,\, \upshape implies\,\,\,\,\,\,}
 \frac{F_{n}(r_{n})}{F_{n}(r_{n}')}\to_{n} 1,
 \label{equhypentest}
\end{align}
(see for instance \eqref{equdefunientestfinvc} and \eqref{equdefunientestfinvcneunet}), then from the first equality in \eqref{equexpazersim} follow the bounds
\begin{align}
& \frac{1}{2^{5}}\frac{1}{B}\frac{1}{\lambda(c-1)+1}(1-\frac{1}{c})\epsilon_{n}(c,\lambda) & \notag\\
\geq &\frac{B}{2^{2}}\frac{(c^{2}-1)\lambda^{2}}{\lambda^{2} c(c+1)((((\lambda-1)n)^{2}+n)^{1/2}+(\lambda-1)n)} & \notag\\
= &\frac{B}{2^{2}}\left(1-\frac{1}{c}\right)\frac{1}{ ((((\lambda-1)n)^{2}+n)^{1/2}+(\lambda-1)n)} & \notag
\end{align}
and an argument similar the one leading to \eqref{equsecsimuppboupro} shows that if $B_{n}\geq 1$, if $(\lambda_{n})_{n}$ is of the form
\begin{align}
\label{equratlamori}
1<\lambda_{n}=1+B_{n}^{2}n^{-1/2},
\end{align}
and if $(c_{n})_{n}$ is bounded away from one, then the second term at the right--hand side of \eqref{equbuoproreg} is upper bounded by
\beql{equsecsimuppbouproref}
&{C_{{n}}}\frac{1}{n^{1/2}}\left({2^{2}}c_{n}(c_{n}+1)+2^{5}\left(\frac{c_{n}}{c_{n}-1}\right)^{3}\left(\log\left(\frac{2(c_{n}+1)(2c_{n}+3)}{\delta}\right) \right.\right. \\&\left.\left. 
+F_{n}\left(\frac{1}{4B_{n}}\left(1-\frac{1}{c_{n}}\right)\left(\frac{1}{n}\right)^{1/2}\right)\right)\right)\eeql 
with $C_{n}\to_{n} 1$.

The  advantage of \eqref{equsecsimuppbouproref} over \eqref{equsecsimuppboupro} relies on the  lower rate of convergence to zero of the radii involved in the covering, which gives room for meaningful estimates under a higher complexity of $(T_{B_{n}}\cG_{n}, Z_{1:n}^{(n)})$. 

Consider for instance the subeuclidean case of order $\alpha \in (0,1)$ (Definition \ref{defeucsubeuc}) in which, assuming that $B_{n}n^{1/2}\to_{n} \infty$,
\begin{align}
\label{equentestspesubsub}
F_{n}(O((B_{n}^{2}n)^{-1/2}))=o(B_{n}^{\alpha}n^{\alpha/2})
\end{align}
as $n\to \infty$. In this case, under the additional assumption that
\begin{align}
\label{equassgrocn}
c_{n}=o(n^{1/4})
\end{align} 
we get that the second term at the right-hand side of \eqref{equbuoproreg} is of the form 
\begin{align}
\label{equbno1}
o(1)+ o(B_{n}^{\alpha}n^{(\alpha-1)/2})
\end{align}
 as $n\to \infty$ for $\delta$ fixed; in particular \eqref{equbno1}, and therefore the second term at the right-hand side of \eqref{equbuoproreg}, goes to zero if some choice $B_{n}=o({n^{\frac{1}{2}(\frac{1}{\alpha}-1)}})$ makes all this possible\footnote{Notice that $B_{n}$ is formally involved on the definition of $F_{n}(\cdot)$.}: this is related to convergence in probability of the regression function in the $L^{2}$ norm (consider the classical i.i.d. case where the  hypotheses are as in \eqref{equdefdiafam}; see also \cite[Theorem 3.14 and Remarks 3.3, 3.5, 3.19]{bargob19}).
 
\medskip

\subsection{Confidence intervals for the $L^{2}$-error with bounded hypotheses}
In the case in which $\cG$ is a set of functions already bounded by $B$, the argument behind the proof of Theorem \ref{theratconlsrpro} can be applied also to the large deviation
\begin{align}
\frac{1}{n}\sum_{k=1}^{n} (\lambda h_{\tilde{g},k}(Z_{k})-\Esp{h_{\tilde{g},k}(Z_{k})}),
\end{align}
whose behavior in probability we did not analyze before. Concretely, we have the following result (we use an alternative description, akin to \eqref{equaltdesdev}, in the statement).

\begin{thm}
\label{corbouproiidreg}
Under the setting in Section \ref{secsetaddnot},  and assuming additionally that $\cG$ is uniformly bounded by $B$ ($|g(x)|\leq B$ for all $(g,x)\in \cG\times \stax$), the inequality
\begin{align}
\frac{1}{n}\sum_{k=1}^{n}||\hat{g}-\phi_{k}||_{k}^{2}\leq& (6\,\lambda-5) \inf_{g\in \cG}\frac{1}{n}\sum_{k=1}^{n}||g-
\phi_{k}||_{k}^{2}
 \notag\\
+ &6\left(   \epsilon_{n}(c,\lambda)\lor (\frac{1}{nb(c,\lambda)}(\log a(c,\lambda,\epsilon_{n}(c,\lambda))+\log(2/\delta)))\right)
\label{equbuoproregsim2}
\end{align}
where $||\cdot||_{k}$ denotes the $L^{2}$ norm in $\calL(S)$ associated to the law of $X_{k}$, holds for every $(c,\lambda,\delta)\in(1,\infty)\times(1,\infty) \times (0,1)$, with probability at least $1-\delta$.
\end{thm}

\begin{proof} First, notice that in this case the truncation operator is indistinguishable from the identity, in particular $\hat{\phi}=\hat{g}$. Let us also introduce
\begin{align}
\label{equintepsg}
\epsilon_{\tilde{g}}\eqdef \frac{1}{n}\sum_{k=1}^{n}\Esp{h_{\tilde{g}}(Z_{k})}=\frac{1}{n}\sum_{k=1}^{n}\Esp{|\tilde{g}(X_{k})-\phi_{k}(X_{k})|^{2}}=\inf_{g\in \cG}\frac{1}{n}\sum_{k=1}^{n}||{g}-\phi_{k}||^{2}_{k}
\end{align}
where the last two equalities follow from   \eqref{equconpyt} and the definition of $\tilde{g}$ in \eqref{equdefgapp}.

By decomposing as in \eqref{inidesprolardev} and using again \eqref{equconpyt} and \eqref{equintepsg}, we arrive at the inequality 
\begin{align}
\frac{1}{n}\sum_{k=1}^{n}||\hat{g}-\phi_{k}||_{k}^{2}\leq &\inf_{g\in\cG}\frac{1}{n}\sum_{k=1}^{n}||{g}-\phi_{k}||^{2}_{k}\notag\\
+& 
\sup_{g\in \cG}\left\{\frac{1}{n}\sum_{k=1}^{n}(\Esp{h_{{g},k}(Z_{k})}-\lambda h_{g,k}(Z_{k}))\right\}+ \frac{1}{n}\sum_{k=1}^{n} (\lambda h_{\tilde{g},k}(Z_{k})-\Esp{h_{\tilde{g},k}(Z_{k})}).
\label{inidesprolardevcon}
\end{align}
We  will seek this time for $\epsilon(\delta)$  such that
\beql{equargspluni}
&\PPro{\sup_{g\in \cG}\left\{\frac{1}{n}\sum_{k=1}^{n}(\Esp{h_{{g},k}(Z_{k})}-\lambda h_{g,k}(Z_{k}))\right\}
\geq \epsilon(\delta) 
}\lor \\&\qqq \PPro{\frac{1}{n}\sum_{k=1}^{n}\lambda h_{\tilde{g}}(Z_{k})-\Esp{h_{\tilde{g}}(Z)}\geq\epsilon(\delta)}
\leq \frac{\delta}{2}
\eeql 
which implies that 
\begin{align}
 \frac{1}{n}\sum_{k=1}^{n}||\hat{g}-\phi_{k}||_{k}^{2}\leq\inf_{g\in \cG}\frac{1}{n}\sum_{k=1}^{n}||g-
\phi_{k}||_{k}^{2}
+2\epsilon(\delta)
\label{equboulasthesim}
\end{align}
with probability at least $1-\delta$.

An  $\epsilon_{1}(\delta)$ appropriate for the first term in the maximization \eqref{equargspluni} can be found exactly as in the proof of Theorem \ref{theratconlsrpro}: it suffices to exchange $\delta$ by $\delta/2$ in that argument. To treat the second term, notice that if
\begin{align}
\label{equconepsiid}
\epsilon>(\lambda-1)\epsilon_{\tilde{g}}
\end{align}
is given, then
\begin{align}
\left\{\frac{1}{n}\sum_{k=1}^{n}(\lambda h_{\tilde{g}}(Z_{k})-\Esp{h_{\tilde{g}}(Z_{k})})>3\epsilon\right\}\notag\\
\subset\left\{\frac{1}{n}\sum_{k=1}^{n}(\lambda h_{\tilde{g}}(Z_{k})-\Esp{h_{\tilde{g}}(Z_{k})})>\epsilon+2(\lambda-1)\epsilon_{\tilde{g}}\right\}\notag\\
=\left\{\frac{1}{n}\sum_{k=1}^{n} (h_{\tilde{g}}(Z_{k})-\Esp{h_{\tilde{g}}(Z_{k})})>\frac{\lambda-1}{\lambda}\left(\frac{\epsilon}{\lambda-1}+\frac{1}{n}\sum_{k=1}^{n}\Esp{h_{\tilde{g}}(Z_{k})}\right)\right\}.
\end{align}
The probability of this last event can be computed following the argument in \cite{bargob19} leading to \eqref{lardevinesym},  this time giving rise to the  inequality 
\begin{align}
\PPro{\left\{\frac{1}{n}\sum_{k=1}^{n}(\lambda \Esp{h_{\tilde{g},k}(Z_{k})}-h_{\tilde{g},k}(Z_{k}))\right\}\geq 3\epsilon} & \notag\\ 
\leq\un_{\{\epsilon< \epsilon_{n}(c,\lambda)\lor (\lambda-1)\epsilon_{\tilde{g}} \}}+& 
a(c)\exp(-b(c,\lambda) n \epsilon )\un_{\{\epsilon\geq \epsilon_{n}(c,\lambda)\lor (\lambda-1)\epsilon_{\tilde{g}}\}}
\label{lardevinesymonefun}
\end{align}
where $a(c)=2(c+1)(2c+3)$ (the covering numbers from the argument in \cite{bargob19} are equal to one). This and the argument in the proof of Theorem \ref{theratconlsrpro} permit to conclude that for any $\epsilon_{2}(\delta)$  satisfying
\begin{align}
\epsilon_{2}(\delta)\geq(\lambda-1)\epsilon_{\tilde{g}}\lor  \epsilon_{n}(c,\delta)\lor\left(\frac{1}{nb(c,\lambda)}(\log a(c)+\log({2}/{\delta}))\right)
\end{align}
we have the estimate
\begin{align}
\PPro{\left\{\frac{1}{n}\sum_{k=1}^{n}(\Esp{h_{\tilde{g},k}(Z_{k})}-\lambda h_{\tilde{g},k}(Z_{k}))\right\}\leq 3\epsilon_{2}(\delta)}\geq 1-\frac{\delta}{2}.
\end{align}
The upper bound at the right hand side of \eqref{equbuoproregsim2} follows from writing down \eqref{equboulasthesim} and using \eqref{equintepsg} when
\begin{align}
\epsilon(\delta)\eqdef 3\left((\lambda-1)\epsilon_{\tilde{g}}+  (\epsilon_{n}(c,\delta)\lor(\frac{1}{nb(c,\lambda)}(\log a(c,\lambda,\epsilon_{n}(c,\lambda))+\log(2/\delta))))\right)
\end{align}
(which is  lower bounded by $\epsilon_{1}(\delta)\lor 3\epsilon_{2}(\delta)$).\fe
\end{proof}

\medskip

Using the inequality 
\begin{align}
(a+b)^{2}\leq\alo{\eta}a^{2}+\alo{1/\eta}
b^{2},
\label{inesumsqu}
\end{align}
valid for every $(a,b,\eta)\in \R\times\R\times (0,\infty)$, where
\begin{align}
\label{equfunpluone}
\alo{\eta}=1+\eta
\end{align}
for every $\eta>0$, it is not difficult to lift the previous result to an estimate of the distance between $\hat{g}$ as before and the conditional expectations of $W$ (without truncation) given $X$, namely:
\bcor
\label{corboucasl2}
Assume that $W_{k}\in L^{2}_{\PP}$ ($k=1,\dots,n$) and consider versions $\Phi_{1:n}$ of the conditional expectation of $W$ given $X$
\begin{align}
\Phi_{k}(X_{k})=\Esp{W_{k}|X_{k}},&\,\, \PP-a.s.,
\end{align}
for $k=1,\dots,n$. Then, with the notation and the hypothesis from Theorem \ref{corbouproiidreg} and the notation \eqref{equfunpluone}, the inequality
\begin{align}
\frac{1}{n}\sum_{k=1}^{n}||\hat{g}-\Phi_{k}||_{k}^{2}\leq& \alo{\eta}\left(\alo{\eta'}(6\,\lambda-5) \inf_{g\in \cG}\frac{1}{n}\sum_{k=1}^{n}||g-
\Phi_{k}||_{k}^{2}
 \right.\notag\\
+ &\left.6\left(\epsilon_{n}(c,\lambda)\lor (\frac{1}{nb(c,\lambda)}(\log a(c,\lambda,\epsilon_{n}(c,\lambda))+\log(2/\delta)))\right)\right)\notag\\
   +&\left(\alo{1/\eta}+\alo{\eta}\alo{1/\eta'}(6\,\lambda-5)\right) \frac{1}{n}\sum_{k=1}^{n}\Esp{(|W_{k}|-B)^{2}\un_{\{|W_{k}|>B\}}}
\label{equbuoproregfur}
\end{align}
holds for every $(c,\lambda,\eta,\eta'\delta)\in(1,\infty)\times(1,\infty) \times (0,\infty)\times(0,\infty)\times (0,1)$, with probability at least $1-\delta$. 
\ecor
\brem
Notice that the bound \eqref{equbuoproregsim2} follows from \eqref{equbuoproregfur} by considering $W\equiv T_{B}W$ and letting $\eta,\eta'\to 0$.\fe
\erems
\begin{proof} (of Corollary \ref{corboucasl2})
For every $(k,g,\eta)\in\{1,\dots,n\}\times \cG\times (0,\infty)$, the estimates
\begin{align}
|| {g}-\Phi_{k}||_{k}^{2}\leq&\alo{\eta}||{g}-\phi_{k}||_{k}^{2}+\alo{1/\eta}||\phi_{k}-\Phi_{k}||_{k}^{2}\notag\\
=&\alo{\eta}||{g}-\phi_{k}||_{k}^{2}+\alo{1/\eta}\Esp{(\Esp{(T_{B}W_{k}-W_{k})|X_{k}})^{2}}\notag\\
\leq&\alo{\eta}||{g}-\phi_{k}||_{k}^{2}+\alo{1/\eta}\Esp{(|W_{k}|-|B|)^{2}\un_{\{|W_{k}|>B\}}}
\label{inebouunbcas}
\end{align}
hold (the last is an application of Jensen's inequality). Taking $g=\hat{g}$ and averaging over $k$ we get an inequality whose right--hand side is   bounded in probability  via \eqref{equbuoproregsim2}. The estimate thus obtained  is then upper bounded again via a further application of \eqref{inebouunbcas} with $\eta'$ in place of $\eta$ and $\phi_{k}$ exchanged with $\Phi_{k}$.\fe
\end{proof}

\subsection*{Acknowledgements}
\label{seccon}

The research leading to this paper was supported by several sources during its different stages, all of which deserve the author's most sincere gratitude. 

At its initial stage in 2020 the author was supported by a grant from the the Chair {\it Stress Test,
RISK Management and Financial Steering}, led by the \'{E}cole Polytechnique (l'X) and its Foundation and sponsored by BNP Paribas, for its participation as postdoctoral researcher at the aforementioned school, and later by a public allocation from the French employment agency (P\^{o}le Emploi) that was very important to keep the pace during a job transition elongated by the COVID-19 pandemic. The research continued in 2021  during the author's affiliation with the \'{E}cole Polytechnique F\'{e}d\'{e}rale de Lausanne (EPFL) for participation in a project that profits also from the results presented above: this memorable stay was possible thanks to a grant provided by the EPFL's Chair of Statistical Data Science, under the supervision of Prof. Sofia Olhede. The paper was finalized during the author's first month as an associate professor at the University of the Andes (Uniandes).

\bibliographystyle{alpha}
\bibliography{estregbib}

\end{document}